\documentclass[letterpaper,10 pt,conference]{IEEEtran}
\IEEEoverridecommandlockouts
\usepackage{amsmath,amssymb,graphics,epsfig,subfigure,adjustbox,bbm,color}

\usepackage{algpseudocode,algorithm,verbatim}
\newcommand{\mc}[1]{\mathcal{#1}}
\newcommand{\te}[1]{\mathrm{#1}}
\newcommand{\field}[1]{{\mathbb{#1}}}
\newcommand{\be}{\begin{equation}}
\newcommand{\ee}{\end{equation}}
\newcommand{\bea}{\begin{eqnarray}}
\newcommand{\eea}{\end{eqnarray}}
\newcommand{\ba}{\begin{array}}
\newcommand{\ea}{\end{array}}
\newcommand{\beas}{\begin{eqnarray*}}
\newcommand{\eeas}{\end{eqnarray*}}
\newcommand{\leftm}{\left[\begin{array}}
\newcommand{\rightm}{\end{array}\right]}

\newtheorem{thm}{\bf Theorem}[section]

\newtheorem{definition}{\it Definition}[section]

\DeclareMathOperator*{\argmin}{arg\,min}
\DeclareMathOperator*{\argmax}{\arg\!\max}

\title{ Distributed Submodular Minimization And Motion Coordination Over Discrete State Space}
\author{Hassan Jaleel and Jeff S. Shamma\thanks{Center for Electrical and Mechanical Science and Engineering, King Abdullah University of Science and Technology, Thuwal, KSA. Email: hassan.jaleel@kaust.edu.sa,
jeff.shamma@kaust.edu.sa.}}
\begin{document}
\maketitle
\begin{abstract}
We develop a framework for the distributed minimization of submodular functions. Submodular functions are a discrete analog of convex functions and are extensively used in large-scale combinatorial optimization problems. While there has been significant interest in the distributed formulations of convex optimization problems, distributed minimization of submodular functions has received relatively little research attention. Our framework relies on an equivalent convex reformulation of a submodular minimization problem, which is efficiently computable. We then use this relaxation to exploit methods for the distributed optimization of convex functions. 
The proposed framework is  applicable to submodular set functions as well as to a wider class of submodular functions defined over certain lattices. We also propose an approach for solving distributed motion coordination problems in discrete state space based on submodular function minimization. We establish through a challenging setup of capture the flag game that submodular functions over lattices can be used to design artificial potential fields over discrete state space in which the agents are attracted towards their goals and are repulsed from obstacles and from each other for collision avoidance. 


\end{abstract}
\section{Introduction}
Submodular functions play a similar role in combinatorial optimization as convex functions play in continuous optimization. These functions can be minimized efficiently in polynomial time using combinatorial or subgradient methods (see e.g. \cite{mccormick2005} and the references therein). Therefore, submodular functions have numerous applications in matroid theory, facility location, min-cut problems, economies of scales, and coalition formation (see e.g., \cite{edmonds1970}, \cite{topkis1978}, \cite{topkis1979}, and \cite{topkis2011}). Sumodular functions can also be maximized approximately, which has applications in resource allocation and welfare problem  \cite{vondrak2008} and \cite{marden2013a}, large scale machine learning problems \cite{bach2010} and \cite{stobbe2010}, controllability of complex networks \cite{summers2016} and \cite{clark2016}, influence maximization \cite{chen2009} and \cite{borgs2014}, and  utility design for multiagent systems \cite{marden2013}. 

Unlike convex optimization and submodular maximization for which efficient distributed algorithms exist in the literature (see e.g., \cite{nedic2010_1}, \cite{mirzasoleiman2013}, \cite{chierichetti2010}, \cite{blelloch2011}, and \cite{lattanzi2011}), distributed minimization of submodular functions has received relatively little research attention. Moreover, most of the existing literature on submodular optimization focuses on submodular set functions, which are defined over all the subsets of a base set. However, we are interested in a wider class of submodular functions defined over ordered lattices, which are products of a finite number of totally ordered sets.

Thus, we propose a framework for the distributed minimization of submodular functions defined over ordered lattices. The enabler in the proposed framework is a particular continuous extension that extends any function defined over an ordered lattice to the set of probability measures. 
This extension, which was presented in \cite{bach2015}, is a generalization of the idea of Lov{\'a}sz extension for set functions \cite{lovasz1983}, and can be computed in polynomial time through a simple greedy algorithm. The key feature of this extension is that the extended function is convex on the set of probability measures if and only if the original function is  submodular. Furthermore, minimizing the extended function over the set of probability measures and minimizing the original function over an ordered lattice are equivalent if and only if the original function is submodular.  

In the proposed framework, we first formulate an equivalent convex optimization problem for a given submodular minimization problem by employing the continuous extension in \cite{bach2015}. After formulating an equivalent convex optimization problem, we propose to implement any efficient distributed optimization algorithm for non-smooth convex functions. {This novel combination of convex reformulation of a submodular minimization problem and distributed convex optimization enables us to solve a submodular minimization problem exactly in polynomial time in a distributed manner.}

Distributed convex optimization is an active area of research  and numerous approaches exist in the literature for the distributed minimization of convex functions (see e.g. \cite{mota2013}, \cite{boyd2011}, \cite{nedic2009}, and \cite{nedic2010}). In the proposed framework, we employ the projected subgradient based algorithm presented in \cite{nedic2010}. This algorithm is well suited for the proposed framework because the subgradient of the continuous extension of a submodular function is already computed as a by product of the greedy algorithm. In the projected subgradient based algorithm, each agent maintains a local estimate of the global optimal solution. An agent is only required to communicate with a subset of the other nodes in the network for information mixing. However, through this local communication and an update in the descent direction of a local subgradient, the algorithm asymptotically drives the estimates of all the agents to the global optimal solution. 

In addition to a distributed framework for submodular minimization, we establish that submodular functions over ordered lattices have an important role in distributed motion coordination over discrete domains. Typically, motion coordination problems under uncertainties are computationally complex. In multiagent systems, the size of the problem increases exponentially with the number of agents, which further increases the complexity of the problem. One approach for handling computational complexity and uncertainties in motion planning is based on artificial potential functions combined with receding horizon control (see \cite{lavalle2006} and the references therein for details).  

We demonstrate that we can design desired potential functions for motion coordination using submodular functions over an ordered lattice. The benefit of designing submodular potential functions is that we have a well established theory of submodular optimization for efficiently minimizing these functions \cite{stobbe2010}, \cite{schrijver2000}, \cite{iwata2001}, and \cite{jegelka2011}. To validate our claim, we consider a version of capture the flag game from \cite{D'Andrea2003} and \cite{chasparis2008}, which is played between two teams called offense and defense. This game is selected because it has a complex setup with both collaborative and adversarial components. Moreover, it offers a variety of challenges involved in multiagent motion coordination. We formulate the problem from the perspective of defense team under the framework of receding horizon control with one step prediction horizon.  

For this game, we design potential functions that generate attractive forces between defenders for cohesion and go to goal behaviors. We also design potential functions that generate repulsive forces for obstacle avoidance and collision avoidance among the members of the defense team. We prove that these potential functions are submodular over the set of decision variables and the problem is a submodular minimization problem. Thus, at each decision time, our proposed framework can exactly solve this problem in polynomial time in a distributed manner. Finally, we show through simulations that the proposed framework for the distributed minimization of submodular functions generates feasible actions for the defense team. Based on the generated actions, the defenders can effectively defend the defense zone while avoiding collisions and obstacles.   

\section{Preliminaries}
\subsection{Notations}\label{subsec:Notations}
Let ${S} = \{s_0,s_1,\ldots,s_{m-1}\}$ be a finite set with cardinality $|{S}|$ and indexed by $\field{Z}_+$, where $\field{Z}_+$ is the set of non-negative integers. A function $f$ is a real valued set function on ${S}$ if for each $A\subseteq {S}$, $f(A) \in \field{R}$. We represent a vector $\te{x} \in \field{R}^n$ as $\te{x} = (x_0,x_1,\ldots,x_{n-1})$. We refer to its $i^{\te{th}}$ component by $\te{x}(i)$, its dimension by $|\te{x}|$, and its Euclidean norm by $\Vert x\Vert$. 
 We define $\{0,1\}^{|S|}$ as a set of all vectors of length $|{S}|$ such that if $\te{x} \in \{0,1\}^{|S|}$, then $\te{x}(i) \in \{0,1\}$ for all $i \in \{0,1,\ldots,|{S}|-1\}$. Similarly, $[0,1]^{|S|}$ is a set of all vectors of length $|{S}|$ such that if $\te{x} \in [0,1]^{|S|}$, then $\te{x}(i) \in [0,1]$ for all $i \in \{0,1,\ldots,|{S}|-1\}$. A unit vector in $\field{R}^n$ is $\te{e}_i$ which is defined as 
\begin{equation}\label{eq:e_i}
\te{e}_{i}(k) = 
\begin{cases}
1 \quad& \text{ if } k = i,\\
0 \quad& \te{otherwise}
\end{cases}
\end{equation}
The indicator vector of a set $A \subseteq {S}$ is $\mathbbm{1}_A$ and is defined as 
\begin{equation}
\mathbbm{1}_A(i) = \begin{cases}
    1       & \quad \text{if } s_i \in A\\
    0   & \quad \text{otherwise}\\
  \end{cases}
\end{equation} 
For any two n-dimensional vectors $\te{x}$ and $\te{y}$, we say that
 $\te{x} \leq \te{y}$ if $x(i) \leq y(i)$ for all $i \in \{0,1,\ldots,n-1\}$. Similarly, 
$\max(\te{x},\te{y})$ and $\min(\te{x},\te{y})$ are vectors in $\field{R}^n$ defined as follows.  
\begin{align*}
\max(\te{x},\te{y}) &= (\max(x(0),y(0)), \ldots,\max(x(n-1),y(n-1)))\\
\min(\te{x},\te{y}) &= (\min(x(0),y(0)),\ldots,\min(x(n-1),y(n-1)))
\end{align*}

 A Partially Ordered Set (POSet) is a set in which the elements are partially ordered with respect to a binary relation ``$\leq$''. Elements $s_i$ and $s_j$ in $S$ are unordered if neither $s_i\leq s_j$ nor $s_j\leq s_i$. A POSet $S$ is a chain if it does not contain any unordered pair, i.e., for any $s_i$ and $s_j$ in $S$, either $s_i \leq s_j$ or $s_j \leq s_i$. 
The supremum and infimum of any pair $s_i$ and $s_j$ are represented as $s_i \vee s_j$ and $s_i \wedge s_j$ respectively. From \cite{topkis1978}, a POSet $S$ is a lattice if for every pair of elements $s_i$ and $s_j$ in $S$
\[
s_i \vee s_j \in S\text{ and } s_i \wedge s_j \in S
\]


\subsection{Submodular Set Functions}\label{def:submod1}
Given a set $S$, a function $f$ is a set function if it is defined over all the subsets of $S$. A set function $f:2^{S} \rightarrow \field{R}$ is submodular if and only if for any two subsets of ${S}$, say $A$ and $B$, the following inequality holds
\begin{equation*}\label{eq:submodDef1}
f(A \cap B)+ f(A \cup B) \leq f(A) + f(B)  
\end{equation*}
However, it is convenient to verify submodularity of a set function through the property of diminishing returns, which is as follows. 
\begin{definition}\label{def:submod2}
{\it A function $f:2^{S} \rightarrow \field{R}$ is submodular if and only if for any two sets $A \subseteq S$ and $B \subseteq S$ such that $A \subseteq B$ and $s_k \in S \setminus B$
\begin{equation*}\label{eq:submodDef2}
f(B \cup \{s_k\}) - f(B) \leq f(A \cup \{s_k\}) - f(A)
\end{equation*}}
\end{definition} 
Thus, submodularity implies that the incremental increase in the value of a function by adding an element in a small set is never smaller than adding that element in a larger set.


\subsection{Submodular Functions Over Lattices}\label{subsec:Submodular Functions over Products}
The notion of submodularity is not restricted to set functions. In this work, we are interested in submodular functions that are real-valued functions defined on set products
\begin{equation*}\label{eq:ProdChain}
\mc{X} = \prod_{i = 0}^{N-1} X_i.
\end{equation*}
In particular, our focus is on set products in which $X_i$ is a lattice for all $i \in \{0,1, \ldots,N-1\}$, and $\te{x} \in \mc{X}$ is a vector, i.e., $\te{x} = (x_0,x_1, \ldots, x_{N-1})$ where $x_i \in X_i$.  
\begin{definition}\label{def:SubmodProdSet}
{\it Let $f$ be a real valued function defined on a lattice $\mc{X}$. Then, $f$ is submodular if and only if for any pair $\te{x}$ and $\te{y}$ in $\mc{X}$}
\begin{equation*}\label{eq:sumodJoinandMeet}
f(\te{x} \vee \te{y}) + f(\te{x}\wedge \te{y}) \leq f(\te{x}) + f(\te{y}) 
\end{equation*}
\end{definition}

Similar to the diminishing return property of submodular set functions, we need a simpler criterion to verify the submodularity of a function defined over a set product. For a function defined over a product of finite number of chains, a simple criterion exists in terms of antitone differences. 	
Let $\mc{X}$ be a product of $N$ chains and $\te{x}$ be an element in $\mc{X}$. Given $A \subset \{0,1,\ldots,N-1\}$, we define a vector $\te{y}_{\te{x}\backslash A}$ as follows.
\[
\begin{cases}
\te{y}_{\te{x}\backslash A} (i) = \te{x}(i) \quad & \text{ if } i \notin A \\
\te{y}_{\te{x}\backslash A} (i) \in X_i \text{ and } \te{y}_{\te{x}\backslash A} (i)> \te{x}(i)	\quad &\text{ if }i \in A
\end{cases}
\]
Therefore, $\te{x} < \te{y}_{\te{x}\backslash A}$ from construction. 
A function $f: \mc{X} \rightarrow \field{R}$ is antitone in $\te{x}(j)$ over $\mc{X}$ if 
	\[
		f(\te{y}_{\te{x}\backslash \{j\}}) \leq f(\te{x})
	\]
for all $\te{x} \in \mc{X}$. Then, from Thm. 3.2 in \cite{topkis1978}, we can verify whether a function defined on a product of finite number of chains is submodular or not  as follows. 
\begin{definition}\label{def:Submod_antitone}
{\it Let $\mc{X} = \prod\limits_{i=0}^{N-1} X_i$ where $X_i$ is a chain for all $i \in \{0,1,\ldots,N-1\}$. A function $f:\mc{X}\rightarrow \field{R}$ is submodular if 
	\begin{align*}
	f(\te{y}_{\te{x}\backslash \{i\}}) - f(\te{x})
	\end{align*}
	is antitone in $\te{x}(j)$ for all $i$ and $j$ in $\{0,1,\ldots,N-1\}$, $i \neq j$, and for all $\te{x} \in \mc{X}$. In other words, the inequality 
\begin{equation}\label{eq:antitone}
	f(\te{y}_{\te{x}\backslash \{i,j\}}) - f(\te{y}_{\te{x}\backslash \{j\}})\leq f(\te{y}_{\te{x}\backslash \{i\}}) - f(\te{x}) 
\end{equation}
should hold for all $i$ and $j$ in $\{0,1,\ldots,N-1\}$, $i \neq j$, and for all $\te{x} \in \mc{X}$
}
\end{definition}

Thus, in the case of chain products, the condition in Def. \ref{def:Submod_antitone} reduces the question of submodularity to comparing all pairs of cross differences. If $X_i \subset \field{Z}$ for all $i \in \{0,1,\ldots,N-1\}$, where $\field{Z}$ is the set of integers,  then Def. \ref{def:Submod_antitone} implies that $f$ is submodular if $f(\te{x} + \te{e}_i) - f(\te{x})$ is antitone, i.e., 
\begin{equation}\label{eq:antitoneInt}
f(\te{x} + \te{e}_j +  \te{e}_i) - f(\te{x} +  \te{e}_j) \leq f(\te{x} +  \te{e}_i) - f(\te{x}) 
\end{equation}
for all $i$ and $j$ in $\{0,1,\ldots,N-1 \}$, $i \neq j$.
 If $X_i$'s are continuous intervals of $\field{R}$, then  the condition in (\ref{eq:antitone}) implies that $f$ is submodular if 
\[
\frac{\partial^2f}{\partial x_i \partial x_j}(\te{x}) \leq 0 
\]
for all $\te{x} \in \mc{X}$ and $i$ and $j$ in $\{0,1,\ldots,N-1 \}$, $i \neq j$.

\section{Submodular Function Minimization}\label{sec:Submod Function Minimization}
We present a brief overview of the tools and techniques for minimizing a submodular function that are relevant to this work. 

\subsection{Minimizing Submodular Set Function}
One approach for solving a combinatorial optimization problem is to formulate a relaxed problem over a continuous set that can be solved efficiently. Every set function $f:2^{S} \rightarrow \field{R}$ can be represented as a function on the vertices of the hypercube $\{0,1\}^{|S|}$. This representation is possible because each $A \subseteq	S$ can be uniquely associated to a vertex of the hypercube via indicator vector $\mathbbm{1}_A$. Consider a function $f$ defined over a set $S = \{s_0,s_1\}$ with 
\[
2^{\mc{S}} = \{\{s_0\},\{s_1\},\{s_0,s_1\},\phi\}.
\]
The indicator vectors are 
\[
\mathbbm{1}_{\{s_0\}}=
\begin{bmatrix}
1\\
0
\end{bmatrix},~
\mathbbm{1}_{\{s_1\}}=
\begin{bmatrix}
0\\
1
\end{bmatrix},~
\mathbbm{1}_{\{s_0,s_1\}}=
\begin{bmatrix}
1\\
1
\end{bmatrix},~
\mathbbm{1}_{\{\phi\}}=
\begin{bmatrix}
0\\
0
\end{bmatrix}.
\] 
Without loss of generality, we can assume that $f(\phi) = 0$.

Since every set function can be defined on the vertices of the unit hypercube $\{0,1\}^{|S|}$, a possible relaxation is to extend the function over the surface of the entire hypercube $[0,1]^{|S|}$. An extension of $f$ on $[0,1]^{|S|}$ is a function defined on the entire surface of the hypercube such that it agrees with $f$ on the vertices of the hypercube. 
A popular extension of a set function defined on the vertices of the hypercube is its convex closure \cite{dughmi2009}. Let $D$ be the set of all distributions over $2^{S}$ and E$_{ A\sim d}f(A) $ be the expected value of $f$ over $2^S$ when the sets $A$  are drawn from $2^S$ according to some distribution $d \in D$. Then the convex closure of $f$ is defined as follows.

\begin{definition}
{\it	Let $f:\{0,1\}^{|S|} \rightarrow \field{R}$ such that $f(\phi) = 0$. For every $\te{x} \in [0,1]^{|S|}$, let $d_f(\te{x})$ be a distribution over $2^S$ with marginal $\te{x}$ such that 
	\begin{equation}\label{eq:d_f}
	d_f(\te{x}) = \argmin_{d \in D(\te{x})} E_{A \sim d}f(A)
	\end{equation}
	where $D(x) \subseteq D$ is the set of all distributions with marginal $\te{x}$. 
	Then the convex closure of $f$ at $\te{x}$ is
	\[
	f_{\te{cl}}(\te{x}) = E_{A \sim d_f(\te{x})}f(A)
	\]}
\end{definition}

It is important to highlight that the extension $f_{\text{cl}}$ is convex for any set function $f$ and does not require $f$ to be submodular.
 
\emph{Example 1}: Let $S = \{s_0,s_1\}$ and 
\[
2^{\mc{S}} = \{\{s_0\},\{s_1\},\{s_0,s_1\},\phi\}.
\]
Consider a function $f:2^S \rightarrow \field{R}$, such that  
\[
f(A) = \min\{|A|,1\}, ~A \subseteq S.
\]
This function is equal to one everywhere except at $A = \phi$, where it is equal to zero, i.e., 
\[
f(\mathbbm{1}_{\{\phi\}}) = 0\text{ and } f(\mathbbm{1}_{\{s_0\}}) = f(\mathbbm{1}_{\{s_1\}})= f(\mathbbm{1}_{\{s_0,s_1\}}) = 1.
\]
To compute the convex closure of $f$,  let
\[
d_1 = (0.1, 0.2, 0.6,0.1)
\]
be a distribution over $2^S$. Then 
\begin{multline*}
E_{A \sim d_1}f(A) = 0.1f(\mathbbm{1}_{\{s_0\}})  + 0.2f(\mathbbm{1}_{\{s_1\}}) + 0.6f(\mathbbm{1}_{\{s_0,s_1\}})\\ + 0.1f(\mathbbm{1}_{\{\phi\}}). 
\end{multline*}

Suppose we want to compute  $f_{\te{cl}}$ at $\te{x} = [0.3 ~~0.5]$. For a distribution $d \in D$ to be in $D(x)$, we need to show that its marginal is $\te{x}$. We first verify whether $d_1$ belongs to $D(\te{x})$ for $\te{x} = [0.3 ~~0.5]$ or not. If $A \in 2^S$ is drawn with respect to $d_1$, then
\begin{align*}
\te{Pr}(s_0 \in A) &= 0.1+0.6 = 0.7 \\
\te{Pr}(s_1 \in A) &= 0.2+0.6 = 0.8 
\end{align*}
Thus, the marginal of $d_1$ is not equal to $\te{x}$. Consider another distribution 
\[
d_2 = (0.1,0.3,0.2,0.4).
\]
Then, the marginal with respect to $d_2$ is
\begin{align*}
\te{Pr}(s_0 \in A) &= 0.1+0.2 = 0.3 \\
\te{Pr}(s_1 \in A) &= 0.3+0.2 = 0.5 
\end{align*}
which is equal to $\te{x}$. Therefore, for any $\te{x} \in [0,1]^{|S|}$, $\te{x}(i)$ can be considered as a probability of $s_i$ to be in a set $A$ that is drawn randomly according to a distribution $d$. 

To compute $f_{\te{cl}}(\te{x})$, we need to compute $d_f(\te{x})$ as defined in (\ref{eq:d_f}). The significance of the extension $f_{\te{cl}}$ is that minimizing a set function $f$ over $2^S$ is equivalent to minimizing $f_{\te{cl}}$ over the entire hypercube $[0,1]^{|S|}$ (see \cite{bach2013} for details). Since $f_{\te{cl}}$  is convex, it can be minimized efficiently.  
However, computing $f_{\text{cl}}$ can be expensive because it involves solving the optimization problem in (\ref{eq:d_f}), which requires $2^{|S|}$ computations.

Another extension of a set function on the surface of the unit hypercube was proposed in \cite{lovasz1983}, which is generally referred to as the Lov{\'a}sz extension. The Lov{\'a}sz extension can be computed  by a simple greedy heuristic as follows. Let $\te{x}=(x_0,x_1,\ldots,x_{|S|-1})$ be a vector in $[0,1]^{|S|}$. Find a permutation 
\[
(i_1,i_2,\ldots,i_{|S|}) \text{ of } (0,1,\ldots,{|S|-1})
\]
such that 
\[
(x({i_1})\geq x({i_2})\geq\ldots\geq x({i_{|S|}})).
\]
Then 
\begin{multline*}\label{eq:lovaszEx}
f_{\te{lv}}(\te{x}) = \sum_{k =1}^{|S|-1}f(\{s_{i_1},\ldots,s_{i_k}\})(x({i_k}) - x({i_{k+1}})) +\\f(S)x(i_{|S|}).
\end{multline*}
{\it The most important result proved in \cite{lovasz1983} was that the Lov{\'a}sz extension $f_{\te{lv}}$ of a set function $f$ is convex if and only if $f$ is submodular. In fact, if $f$ is submodular, its convex closure $f_{\te{cl}}$ and the Lov{\'a}sz extension $f_{\te{lv}}$ are the same.}

Thus, the fundamental result regarding the minimization of a submodular set function, as presented succinctly in Prop. 3.7 of \cite{bach2013}, is that minimizing the Lov{\'a}sz extension $f_{\te{lv}}$ of a set function $f$ on $[0,1]^{|S|}$ is the same as minimizing $f$ on $\{0,1\}^{|S|}$, which is the same as minimizing $f$ on $2^{|S|}$. In other words, the following three problems are equivalent. 
\begin{align*}
&\min\{f(A):A \subseteq S \}\\
&\min\{f(X):X \in  \{0,1\}^{|S|}\}\\
&\min\{f_{\te{lv}}(\te{x}):\te{x} \in  [0,1]^{|S|}\}
\end{align*}
{\it This equivalence implies that a submodular minimization problem in discrete domain can be solved exactly by solving a convex problem in continuous domain for which efficient algorithms exist like sub-gradient based algorithms.} 

Furthermore, let $\partial f_{\te{lv}}\bigr\rvert_{\te{x}}$ be a subgradient vector of $f_{\te{lv}}$ evaluated at $\te{x}$. Then, $\partial f_{\te{lv}}$ can also be computed via the same greedy heuristic while computing the Lov{\'a}sz extension as follows. 
\begin{equation*}
\partial f_{\te{lv}}(k)\bigr\rvert_{\te{x}} = f(\{i_1,\ldots,i_k\}) - f(\{i_1,\ldots,i_{k-1}\}).
\end{equation*} 
for all $k \in \{0,1,\ldots,|S|-1\}$, where $(i_1,i_2,\ldots,i_{|S|})$ is the permutation of $(0,1,\ldots,{|S|-1})$ that was used for computing the Lov{\'a}sz extension.

\subsection{Submodular Minimization Over Ordered Lattices}
In \cite{bach2015}, it was shown that most of the results relating submodularity and convexity like efficient minimization via Lov{\'a}sz extension can be extended to submodular functions over lattices. In particular, lattices defined by product of chains was considered and an extension was proposed in the set of probability measures. It was proved that the proposed extension on the set of probability measures was convex if and only if the original function defined over product of chains was submodular. Moreover, it was proved that minimizing the original function was equivalent to minimizing the proposed convex extension on the set of probability measures. 

A greedy algorithm was also presented in \cite{bach2015} for computing the continuous extension of a submodular function defined over a finite chain product. In the second half of this paper, we show that a class of motion coordination problems over discrete domain can be formulated as submodular minimization problems over chain products. Hence, the greedy algorithm in \cite{bach2015} can play a significant role in efficient motion coordination over discrete domain for multiagent systems. Therefore, we include the algorithm here for the completeness of presentation. For details, we refer the readers to \cite{bach2015}.  
%

Let $\mc{X} $ be a product of $N$ discrete sets with finite number of elements 
\[
\mc{X} = \prod_{i=0}^{N-1} X_i.
\]
We assume that $X_i = \{s_0,s_1,\ldots,s_{m_i-1}\}$ is a chain for all $i$, which implies that the product set $\mc{X}$ is a lattice. Since $X_i$'s are chains, we can order their elements and represent each set by the index set
\[
X_i = \{0,1,\ldots, m_i-1\}.
\]
such that $s_j \leq s_{j+1}$. Then, any $\te{x}\in \mc{X}$ will be an index vector.    

We are interested in computing an extension of a function $f$ over a continuous space. For set functions, the continuous space was the entire hypercube $[0,1]^{|S|}$ and $\te{x}(i)$ was interpreted as a probability measure on the set $\{0,1\}$ corresponding to the entry $s_i \in S$. In the case of set products in which every set contains more than a single element, the notion of probability measure needs to be more general. 

Let $P(X_i)$ be the set of all probability measures on ${X}_i$. Then, $\mu_i \in P(X_i)$ is a vector in $[0,1]^{m_i}$ such that
\begin{align*}
\sum_{j=0}^{m_i-1}&\mu_{i}(j) =1
\end{align*}
For a product set $\mc{X}$, let $\mc{P}(\mc{X})$ be the set of product probability measures. For any $\mu \in \mc{P}(\mc{X})$
\[
\mu = \prod_{i=0}^{N-1}\mu_i, ~\mu_i \in P(X_i) ~~\forall~~ i
\]
A probability measure $\mu_{i}$ is degenerate if $\mu_i(j) =1$ for some $j\in \{0,1,\ldots,m_i-1\}$. 
We define $F_{\mu_{i}}:X_i \rightarrow \field{R}$ as 
\begin{equation*}
F_{\mu_{i}}(j) = \sum_{l = j}^{m_i-1} \mu_{i}(l).
\end{equation*}
Thus, $F_{\mu_{i}}$ is similar to cumulative distribution function but is reverse of it. Since $\mu_{i}$ is a probability measure, $F_{\mu_{i}}(0) $ is always equal to one. Therefore, we will ignore $F_{\mu_{i}}(0) $ and only consider $m_i-1$ values to reduce dimension of the problem. 

For a probability measure $\mu_{i}$ on $X_i$, we define a vector $\rho_{{i}}$ as 
\begin{equation}\label{eq:rho_i}
\rho_{{i}} = (F_{\mu_{i}}(1), F_{\mu_{i}}(2),\ldots,F_{\mu_{i}}(m_i-1)).
\end{equation}
Since
\[
F_{\mu_{i}}(j+1) \leq F_{\mu_{i}}(j),
\]
for all $j \in \{0,1,\ldots,m_i-1\}$, $\rho_{{i}}$ is a vector with non-increasing entries. The equality $\rho_{{i}}(j) = \rho_{{i}}(j+1)$ occurs if and only if $\mu_{i}(j) = 0$. Thus, $\rho_i \in [0,1]^{m_i-1}_{\downarrow}$ where
\[
[0,1]^{m_i-1}_{\downarrow} = \{\tilde{\rho}\in [0,1]^{m_i-1}  ~:~ \tilde{\rho}(i+1)\leq \tilde{\rho}(i)~\forall~ i\}
\]
 For a product set $\mc{X}$, we define the set $\Omega(\mc{X})$ as
 \begin{equation}\label{eq:Omega}
 \Omega(\mc{X}) = \prod_{i=0}^{N-1}[0,1]^{m_i-1}_{\downarrow} 
 \end{equation} 
 Then any $\rho \in \Omega(\mc{X})$ is 
\[
\rho =\prod_{i=0}^{N-1} \rho_{{i}}, \text{ where } \rho_i \in [0,1]^{m_i-1}_{\downarrow}
\]

Let $\theta_{\rho_i}:[0,1] \rightarrow \{0,1,\ldots,m_i-1\}$ be an inverse map of $\rho_{{i}}$ and is defined as 
\begin{equation*}
\theta_{\rho_i}(t) = \max\{0\leq l \leq m_i-1 : \rho_{{i}}(l) \geq t\}.
\end{equation*}
From the definition of $\theta_{\rho_i}$,   
\[ \theta_{\rho_i}(t) =
\begin{cases}
m_i-1  & \quad t < \rho_{{i}}(m_i-1).  \\
l       & \quad  \rho_{{i}}(l+1)<t<\rho_{{i}}(l),\\
& \quad l\in\{1,2,\ldots,m_i-2\}.\\
0  & \quad  t> \rho_{{i}}(1).\\
\end{cases}
\]
The boundary values can be arbitrary and does not impact the overall setup. The definition of $\theta_{\rho_i}$ is extended to a product set $\mc{X}$ as follows
\[
\theta_{\rho}(t) = \prod_{i=0}^{N-1} \theta_{\rho_i}(t) ,
\]
where $t \in [0,1]$. 

\begin{algorithm}[h]
	\caption{\bf Greedy Algorithm }\label{alg:GreedyAlg}
	\begin{algorithmic}[1]
		\Require $\rho =  \prod_{i=0}^{N-1} \rho_{{i}}$ . 
		\State Form a set ${Q}$ as follows. 
		\begin{align*}
		{Q} &= \{\rho_{{0}}(1),\ldots,\rho_{{0}}(m_0-1),\rho_{{1}}(1),\ldots,\\
		&\rho_{{1}}(m_1-1),
		\ldots,\rho_{{{N-1}}}(1),\ldots,\rho_{{{N-1}}}(m_{N-1}-1)\}.
		\end{align*}
		The number of elements in ${Q}$ is $r = \sum_{i = 0}^{N-1} m_i - n$. In the summation, $n$ is subtracted because $F_{\mu_{i}}(0)$ is neglected in $\rho_{{i}}$ for all $i$.
		\State Arrange all the $r$ values of  ${Q}$ in decreasing order in the set ${Q}_{\te{dec}}$, i.e.,  
		\begin{align*}
		{Q}_{\te{dec}} = \{\rho_{{{i_1}}}(j_1), \rho_{{{i_2}}}(j_2), \ldots,\rho_{{{i_r}}}(j_r)\},
		\end{align*}
		such that 
		$$\rho_{{{i_1}}}(j_1) \geq \rho_{{{i_2}}}(j_2)\geq \ldots \geq \rho_{{{i_r}}}(j_r).$$ The ties are handled randomly. However, in the case of ties within $\rho_{{i}}$ for some $i$, the order of the values are maintained.  
		\State Compute the extension of function $f$ over the probability measures  as follows
		\begin{equation}\label{eq:general_LovExt}
		f^{\te{ext}}(\rho) = f(0) + \sum_{s = 1}^{r} \te{t}(s)\left( f(\te{y}_s) - f(\te{y}_{s-1})\right),
		\end{equation}
		where 
		\begin{align*}
		\te{t}(s) &= \rho_{{{i_s}}}(j_s)~~\forall~~ s\in \{0,1,\ldots,r-1\}
		\end{align*}
		Moreover, the vector $\te{y}_s \in \mc{X}$ is  
		\begin{equation}
		\te{y}_s=
		\begin{cases}
			(0,0,\ldots,0) \quad & s = 0 \\
			\te{y}_{s-1} + e_{i_s} \quad & 1\leq s\leq r-1 \\
			(m_0-1,m_1-1,\ldots,m_{N-1}-1)& s = r
		\end{cases}
		\end{equation}
		\State The subgradient of $f^{\te{ext}}$ evaluated at $\rho$ is 
		\[
		\partial f^{\te{ext}}\bigr\rvert_{\rho} =\prod_{i=0}^{N-1} \partial f^{\te{ext}}\bigr\rvert_{\rho_{i}}. 
		\]
		The $j^{\mathrm{th}}$ component of  $ \partial f^{\te{ext}}\bigr\rvert_{\rho_{i}} $ is 
		\begin{equation}\label{eq:general_subgrad}
		\partial f^{\te{ext}}\bigr\rvert_{\rho_{i}}(j) = f(\te{y}_g) - f(\te{y}_{g-1}),
		\end{equation}
		where $g = \min\{s \in \{0,1,\ldots,r-1\}: \te{y}_s(i) = j\}$. 
	\end{algorithmic}
\end{algorithm}

Let $f:\mc{X}\rightarrow \field{R}$ be a real valued function defined over $\mc{X}$. Then, the greedy algorithm for computing an extension of $f$ over a continuous space is presented in Alg. \ref{alg:GreedyAlg}. The extension $f^{\te{ext}}$ of $f$ is given in  (\ref{eq:general_LovExt}) and the subgradient of $f^{\te{ext}}$ is in  (\ref{eq:general_subgrad}). The algorithm requires sorting $r$ values, which has a complexity of $O(r \log r)$, and $r$ evaluations of the function, where 
\begin{equation}\label{eq:r}
r = \sum_{i=0}^{N-1}m_i - N.
\end{equation} 
We refer the reader to \cite{bach2015} for the details and the complexity analysis of the greedy algorithm.  

It was proved in \cite{bach2015} that for a function $f:\mc{X} \rightarrow \field{R}$, where $\mc{X}$ is a product of $N$ finite chains, $f^{\te{ext}}(\rho)$ is convex if and only if $f$ is submodular. It was also proved that minimizing $f$ over $\mc{X}$ is equivalent to minimizing $f^{\te{ext}}$ over $\Omega(\mc{X})$, i.e, 
\begin{align*}
\min_{\te{x} \in \mc{X}} ~~f(\te{x}) = \min_{\rho \in \Omega(\mc{X})} f^{\te{ext}}(\rho)
\end{align*} 
and $\rho^*\in \Omega(\mc{X})$ is the minimizer for $f^{\te{ext}}$ if and only if $\theta_{\rho^*}(t)$ is a minimizer for $f$ for all $t \in [0,1]$. 
Therefore, by minimizing $f^{\te{ext}}$ over $\Omega(\mc{X})$, we can find a minimizer for a submodular function $f$ over an ordered lattice $\mc{X}$. 

For a submodular set function $f:2^S \rightarrow \field{R}$
\[
\min_{A \subseteq S} ~~f(A) = \min_{\te{x} \in \mc{X}} ~~f(X)
\]
where $\mc{X} = \prod\limits_{i = 0}^{|S|-1} X_i$ and $X_i = \{0,1\}$ for all $i$. Thus, submodular set functions are a particular instance of submodular functions over product sets. Moreover, Alg. \ref{alg:GreedyAlg} reduces to Lov{\'a}sz extension for set functions. Therefore, from this point onwards, we will only focus on submodular functions defined over product of chains. 
\section{Distributed Submodular Minimization}\label{sec:Distributed Submodular Minimization}
In this section, we present the main contribution of this work, which is a distributed algorithm for minimizing a submodular function defined over a product of $N$ chains. 

Consider a system comprising $N$ agents, $\{a_0,a_1,\ldots,a_{N-1}\}$. The global objective is to minimize a cost function, which is the sum of $N$ terms over a product set $\mc{X}$. Each agent has information about one term only in the global cost function. Thus, the agents need to solve the following optimization problem collaboratively
\begin{align}\label{prob:distProb}
\min_{\te{x} \in \mc{X}}~~~~~~ J(\te{x})= \sum_{i=0}^{N-1} J_i(\te{x}) \tag{$\mc{P}$} 
\end{align}
where
\begin{align*}
\mc{X} = \prod_{j=0}^{p-1}X_j,~~~~|X_j| = m_j\nonumber
\end{align*}
We assume that $J_i: \mc{X} \rightarrow \field{R}$ is a submodular function and each $X_j$ is a chain. Since the total cost is a sum of $N$ submodular functions, it is also a submodular function. 

The cost function of each agent in \ref{prob:distProb} depends on the entire decision vector, which is global information. However, we assume that each agent has access to local information only. The local information of agent $a_i$ consists of the term $J_i$ in the cost function. Moreover,  each agent is allowed to communicate with a subset of other agents in the network. Therefore, no agent has direct access to any global information. The communication network is represented by a graph $\field{G}(V,\mc{E})$, where $V = \{a_0,a_1,\ldots,a_{N-1}\}$ is the set of vertices and $\mc{E} \subseteq V\times V$ is the set of edges. An edge $(a_i,a_j)\in \mc{E}$ implies that agents $a_i$ and $a_j$ can communicate with each other. The closed neighborhood set of $a_i$ contains $a_i$ and the agents with which $a_i$ can communication, i.e., 
$${N}(a_i) = a_i \cup \{a_j \in V : (a_i,a_j) \in \mc{E}\}.$$ The communication network topology is represented algebraically by a weighted incidence matrix $A$ defined as follows. 
\[
A(i,j) = 
\begin{cases}
c_{ij} \quad & a_j \in N(a_i)\\
0	\quad  & \te{otherwise}	
\end{cases}
\] 
where $c_{ij}\geq0$ for all $i$ and $j$. 
\vspace{0.1in}
\begin{thm}
{\it
	If the communication network topology satisfies the following conditions
	\begin{enumerate}
		\item The communication graph $\field{G}(V,\mc{E})$ is strongly connected. 
		\item There exists a scalar $\eta \in (0,1)$ such that $c_{ii} \geq \eta$ for all $i\in \{0,1,\ldots,N-1\}$. 
		\item For any pair of agents $(a_i,a_j) \in \mc{E}$, $c_{ij}\geq \eta$. 
		\item Matrix $A$ is doubly stochastic, i.e., $\sum_{i=0}^{N-1} c_{ij} = 1$ and $\sum_{j=0}^{N-1} c_{ij} = 1$ for all $i$ and $j$ in $\{0,1,\ldots,N-1\}$
	\end{enumerate}
Then, the submodular function in \ref{prob:distProb} can be minimized exactly in a distributed manner. }
\end{thm}
\begin{proof}
To prove this theorem, we rely on the relaxation based approach for solving submodular minimization problems as presented in the previous section. The first step is to formulate the following relaxed problem. 
\begin{align}\label{prob:relaxedProb}
 \min\limits_{\rho \in \Omega(\mc{X})}& ~~J^{\te{ext}}(\rho)= \sum_{i=0}^{N-1} J_i^{\te{ext}}(\rho). \tag{$\mc{P}1$}
\end{align}
In the relaxed problem, $J^{\te{ext}}$ is the extension of $J$, which is computed by (\ref{eq:general_LovExt}) of Alg. \ref{alg:GreedyAlg}, and $\Omega(\mc{X})$ is the constraint set defined in (\ref{eq:Omega}). 
It was shown in \cite{bach2015} that the problems \ref{prob:distProb} and \ref{prob:relaxedProb} are equivalent. Therefore, we can find an optimal solution to \ref{prob:distProb} through an optimal solution to \ref{prob:relaxedProb}.

Problem \ref{prob:relaxedProb} is a constrained convex optimization problem. Based on the results in \cite{nedic2010}, if the communication network topology satisfies the conditions 1-4 in the theorem statement, then the consensus based projected subgradient algorithm presented in \cite{nedic2010} can solve \ref{prob:relaxedProb}. In the algorithm presented in \cite{nedic2010}, only neighboring agents are required to communicate with each other. Therefore, we can solve \ref{prob:relaxedProb} distributedly through consensus based projected subgradient algorithm, which concludes the proof the theorem. The details of the algorithm are presented in Alg. \ref{alg:DistOpt}. 
\end{proof}

Since the cost of each agent depends on the entire state vector, agents need global information to solve the optimization problem. The main idea in the distributed consensus based algorithm is that each agent generates and maintains an estimate of the optimal solution based on its local information and communication with its neighbors. The local information of agent $a_i$ is the cost function $J_i$. It solves a local optimization problem and exchanges its local estimate of the solution with its neighbors. Then, it updates its estimate of the optimal solution by mixing the information it received from its neighbors. Under the conditions mentioned in the theorem statement, it was proved that this local computation and communication converges to a globally optimal solution. The details are presented in  Alg. \ref{alg:DistOpt} from the perspective of agent $a_i$. 

In Alg. \ref{alg:DistOpt}, $a_i$ starts by initializing its estimate $\rho^i$ of the optimal solution with a feasible product vector, i.e., $\rho^i[0] \in \Omega_{\mc{X}}$. 
To update $\rho^i[k]$ for all $j \in \{0,1, \ldots,N-1\}$, $a_i$ exchanges its local estimate with all the agents in its neighborhood set $N(a_i)\backslash a_i$. The estimates are updated in two steps, a consensus step and a gradient descent step. The consensus step is  (\ref{eq:Consensus}) in which $a_i$ computes a weighted combination of the estimates of $a_k \in N(a_i)$ by assigning weight $c_{iw}$ to $\rho^w[k-1]$  . The gradient descent step is in  (\ref{eq:Update}), in which the combined estimate $\nu^i$ is updated in the direction of gradient descent of $J^{\te{ext}}_i$ evaluated at $\nu^i$. Here, $\gamma_k$ is the step size of the descent algorithm at time $k$. The gradient $\partial J_i^{\te{ext}}\bigr\rvert_{\nu^i}$ is computed through the greedy algorithm. 

\begin{algorithm}[t!]
\caption{\bf Distributed Submodular Minimization }\label{alg:DistOpt}
\begin{algorithmic}[1]
\Statex To solve \ref{prob:relaxedProb}, agent $a_i$ has to perform the following steps:

\State Select any $\rho \in \Omega(\mc{X})$, where $\Omega(\mc{X})$ is defined in (\ref{eq:Omega}). Set  
\begin{align*}
\rho^i[0] &= \rho,
\end{align*}
\State At each time $k$, update $\rho^i[k-1]$ as follows  
\For{$k=1$ to iter} 
\For{$j=0$ to $p-1$}
\begin{equation}\label{eq:Consensus}
\nu^i_j = \sum_{w=0}^{N-1} c_{iw}\rho^w_j[k-1]
\end{equation}
\EndFor
\State Set
\begin{equation}\label{eq:Update}
\rho^i [k] = \field{P}_{\Omega({\mc{X}})}\left(\nu^i - \gamma_k \partial J_i^{\te{ext}}\bigr\rvert_{\nu^i}\right),
\end{equation}
where $\nu^i =\prod\limits_{j=0}^{p-1} \nu^i_j$, and $\gamma_k$ is the step size. 
\EndFor
\State Set $\hat{\rho}^i =\rho^i[\te{iter}]$. 

\State Agent $a_i$'s estimate of optimal solution is
\begin{equation}\label{eq:distoptSolution}
\hat{\te{x}}^i = \theta_{\hat{\rho}^i}(\hat{t})
\end{equation}
for some $\hat{t} \in [0,1]$. 

\end{algorithmic}
\end{algorithm}

Finally, $ \field{P}_{\Omega(\mc{X})}(\xi^i)$ is the projection operator that projects $\xi^i$ on the constraint set $\Omega(\mc{X})$. Let 
\begin{align*}
\xi^i= \nu^i- \gamma_k \partial J_i^{\te{ext}}\bigr\rvert_{\nu^i},~ 
\end{align*}
Since $\nu^i$ and $\partial J_{i}^{\te{ext}}$ are product vectors, 
\begin{align*}
\xi^i &= \prod_{j=0}^{p-1} \xi^i_j, ~~j \in \{0,1,\ldots,p-1\}\\
\xi^i_j &=\nu^i_j - \gamma_k \partial J_i^{\te{ext}}\bigr\rvert_{\nu^i_j}
\end{align*}
Thus, the projection of $\xi^i$ on $\Omega(\mc{X})$ can be decomposed into projecting each $ \xi^i_{j}$ on $[0,1]^{m_j-1}_{\downarrow} $ for which we solve the following problem.  
\begin{align}\label{eq:ProjectionOptimization}
\min_{\tilde{\rho} \in [0,1]^{m_j-1}} \Vert \tilde{\rho} &- \xi^i_{j}\Vert^2 \\
\te{ s.t. }~~~~~~~~ C\tilde{\rho}&\leq  0_{m_j-2}	\nonumber
\end{align}
where $0_{m_j-2} \in \field{R}^{m_j-2}$ with all entries equal to 0 and $C \in \field{R}^{(m_j-2)\times (m_j-1)}$ with entries equal to
\[
C_{uv} = 
\begin{cases}
-1 \quad & \te{if } u=v \\
1  \quad & \te{if } v=u+1\\
0  \quad & \te{otherwise}
\end{cases}
\]
The inequality constraints ensure that the solution to (\ref{eq:ProjectionOptimization}) has non-increasing entries, i.e., 
\[
\tilde{\rho}(u+1)-\tilde{\rho}(u) \leq 0 ~~\forall~~ u \in \{0,1,\ldots,m_j-2\}
\]

The vector $\rho^i$ is updated through Eqs (\ref{eq:Consensus}) and (\ref{eq:Update}) for $\te{iter}$ number of iterations. In \cite{nedic2010}, it was shown that 
\[
\rho^i[\te{iter}] \rightarrow \rho^*, \text{ as } \te{iter}\rightarrow \infty .
\]
where $\rho^*$ is the optimal solution to \ref{prob:relaxedProb}. Let $\mc{X}^* \subseteq \mc{X}$ be the set of optimal solutions for \ref{prob:distProb}. Then, $\theta_{\rho^*}(t)\in \mc{X}^*$ for all $t \in [0,1]$. Let $t^i\in[0,1]$ be the value used by agent $i$ to compute its optimal solution $\theta_{\rho^*}(t^i)$. If $|\mc{X}^*| = 1$, i.e., \ref{prob:distProb} has a unique optimal solution $\te{x}^*\in 
\mc{X}$, then $\theta_{\rho^*}(t^i) = \te{x}^*$ for all $i$. However, if $|\mc{X}^*|> 1$, $t^i$ and $t^j$ can lead to different elements in $\mc{X}^*$ for $t^i\neq t^j$. 
Therefore, if there is an additional constraint that all the agents should select the same optimal solution, we need to set
\[
t^i = \hat{t} \text{ for all } i \in \{0,1,\ldots,N-1\} 
\]
for some $\hat{t} \in [0,1]$. For practical purposes, Alg. \ref{alg:DistOpt} is executed for a limited number of iterations. In such a scenario, each agent computes its own estimate of the optimal solution in (\ref{eq:distoptSolution}) at the specified $\hat{t}$.  

In Alg. \ref{alg:DistOpt}, there are two primary operations that an agent performs in every iterations. The first operation is the computation of subgradient, which is computed through Alg. \ref{alg:GreedyAlg}. The complexity of this algorithm was already discussed in the previous section. 
The second operation is the projection of the updated estimate of the optimization vector on the constraint set by solving (\ref{eq:ProjectionOptimization}). This is an isotonic regression problem  and can be solved by using any efficient optimization solver package. 

Next, we present our second main contribution, which is to show that a class of distributed motion coordination problems over discrete domain can be formulated as submodular minimization problems over chain products.
\section{Distributed Motion Coordination over Discrete Domain}
 
As outlined in the classical ``boids'' model in \cite{reynolds1987}, the motion of an individual agent in a multiagent system should be a combination of certain fundamental behaviors. These behaviors include collision avoidance, cohesion, and alignment. Cohesion corresponds to the tendency of the agents to remain close to each other, and alignment refers to the ability of the agents to align with a desired orientation and reach a desired goal point. In addition to these behaviors, agents should be able to avoid any obstacles in the environment. 

We demonstrate that the behaviors in the ``boids'' model can be modeled as submodular minimization problems over discrete domain. In particular, we design potential functions whose minima correspond to the desired behaviors like reaching a particular point, obstacle avoidance, and collision avoidance. Then, we prove that these potential functions are submodular over a lattice of chain products. The advantage of using submodular functions for designing potential functions is that their subgradient can be computed in polynomial time through Alg. \ref{alg:GreedyAlg}. Therefore, we can minimize these functions efficiently through subgradient descent algorithms. Moreover, the framework proposed in this work can be employed for the distributed minimization of the potential functions. 

We establish the effectiveness of submodular functions in distributed motion planning and coordination problems through an example setup, which is inspired from capture the flag game as presented in \cite{D'Andrea2003} and \cite{chasparis2008}. Capture the flag is a challenging setup that involves two teams of agents competing against each other. The members of the same team need to collaborate with each other to devise a mobility strategy that can stop the opposing team from achieving their objective. Thus, the setup has both collaborative and adversarial components involved in decision making, which makes it a challenging problem even for simple cases. 
We want to highlight that our objective is not to provide a solution to the well studied capture the flag game. Instead, our objective is to show that the proposed framework can be used effectively for such complex motion coordination problems over a discrete state space. 

\subsection{Problem Formulation}
The game is played between two teams of agents, offense and defense, over a time interval of length $T$. We will refer to the members of the offense and defense teams as attackers and defenders respectively. The arena is a square region of area $N_g^2$ that is discretized into a two dimensional grid having $N_g \times N_g$ sectors as shown in Fig. \ref{fig:Arena}. The discretized arena is represented by a set $\mc{G} = G \times G$, which is an integer lattice, i.e., each $z = (x,y)$ in $\mc{G}$ is a vector in $\field{R}^2$, where $x$ and $y$ belong to $G = \{0,1,2,\ldots,N_g-1\}$.  

A flag is assumed to be placed in the arena and the area surrounding it is declared as a defense zone. The defense zone ${D}=\{z^f_0,z^f_1,\ldots,z^f_{n_f-1}\}$ is a subset of $\mc{G}$ with $n_f$ points. The points in $D$ are stacked in a vector 
\[
\te{z}^f=(z^f_0,z^f_1,\ldots,z^f_{n_f-1}),~~~\te{z}^f \in \field{Z}^{2n_f}
\]
in which each $z^f_i =(x^f_i,y^f_i)$ is a point in $\field{Z}^2$.  
The grid points of the shaded region at the top of Fig. \ref{fig:Arena} comprise the defense zone. The objective of the offense is to capture the flag. The flag is considered captured if any attacker reaches a point in the defense zone. Once the flag is captured, the game stops and the offense wins. On the other hand, the objective of the defense is to stop the attackers from entering the defense zone either by capturing them or forcing them away. To defend the defense zone, there needs to be collaboration and cohesion among the defenders. An attacker is in captured state if its current location is shared by at least one defender. However, if that defender moves to a different location, the state of the attacker switches from captured to active. If no attacker can enter the defense zone for the duration of the game, the defense wins.

Let $P = \{a,d\}$ be a set of teams where $a$ and $d$ correspond to the teams of attackers and defenders respectively. 
Let $n_p$ be the number of players and $p_i$ be the $i^{\te{th}}$ player in team $p \in P$. The locations of all the players in a team at time $k$ are stacked in a vector $\te{z}^{p}(k)\in \field{Z}^{2n_p}$ where the location of $p_i$ at time $k$ is $z^p_i(k) =(x_i^{p}(k) ,y_i^{p}(k) )$.
\begin{figure}[t!]
	\centering
	\includegraphics[trim= 10cm 2cm 0cm 2.5cm,clip, scale=0.3]{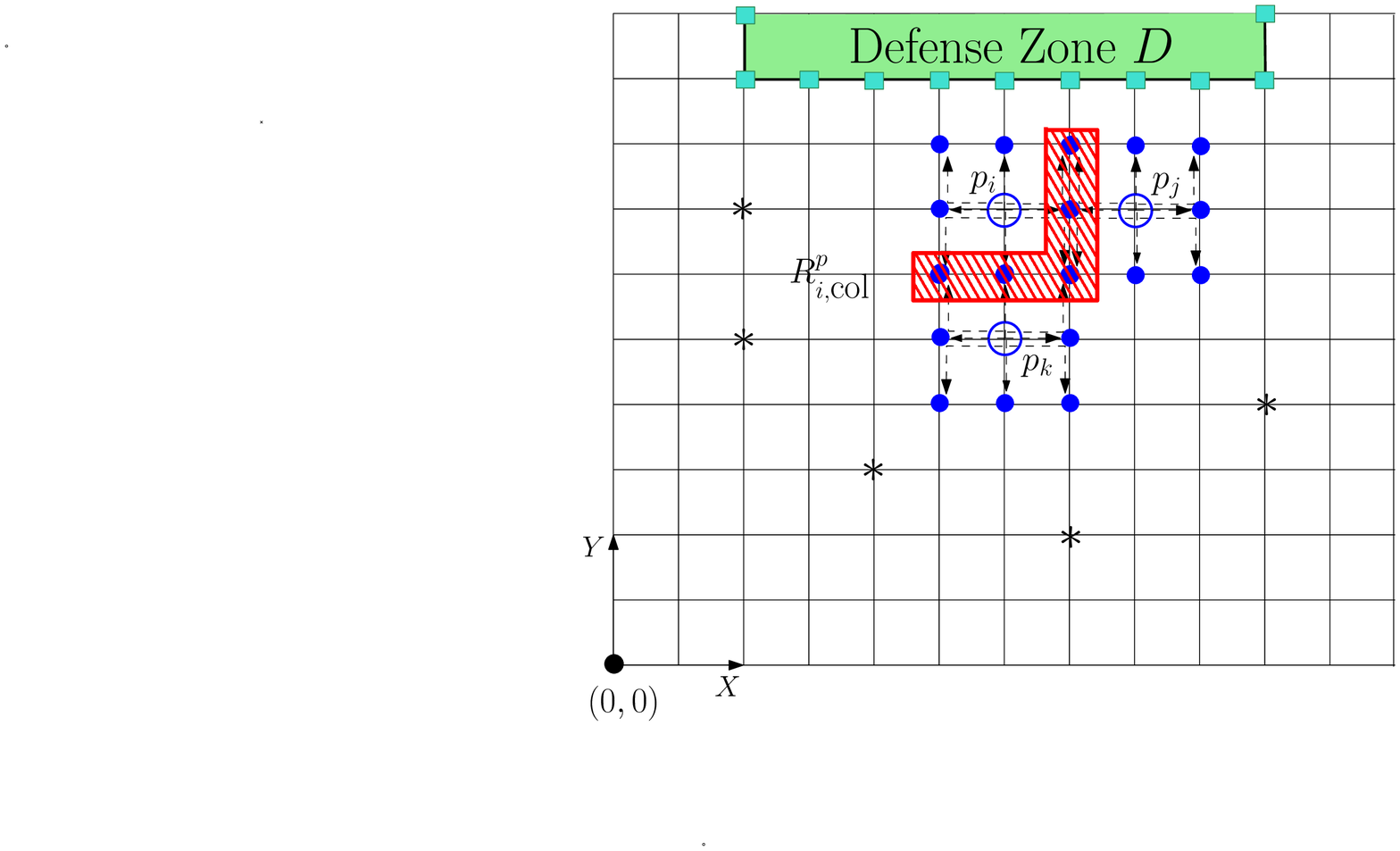}
	\caption{Layout of the playing arena. }
	\label{fig:Arena}
\end{figure}
The update  equation for $p_i$ is
\begin{align*}
z^p_i(k+1) &= z^p_i(k) + {u^p_{i}}(k), 
\end{align*}
where $u^p_{i}(k) = ( u^p_{x,i}(k), u^p_{y,i}(k)) \in {U}^p_x \times {U}^p_y$,
\begin{align*}
{U}_x^p & = \{0,\pm1,\ldots,\pm u_x^{\max}\},\\
{U}_y^p & = \{0,\pm1,\ldots,\pm u_y^{\max}\}.
\end{align*} 
Let $U^p_i = U^p_{x,i} \times U^p_{y,i}$ be the input set of player $p_i$. Then, the reachable set of $p_i$ at time $k$ is 
\begin{align*}
{R}^p_i(k) =\{z \in \mc{G}: z = z^p_i(k) + u^p_{i}, ~~u^p_{i} \in U^p_i\},
\end{align*}
i.e., ${R}^p_i(k)$ is the set of all points that $p_i$ can reach in one time step. For notational convenience, we will drop $k$ from the arguments. 
We assume that all the players are homogeneous, i.e., $U^p_i = U\times U$ for all $p \in P$ and $i \in \{0,1,\ldots,n_p\}$ where $U \subset \field{Z}$.
The reachable sets of players with $u_x^{\max} = u_y^{\max}= 1$ are depicted in Fig. \ref{fig:Arena}. 

Two players can collide if their reachable sets overlap with each other. Let 
\begin{align*}
{R}^p_{i,\te{col}} =\bigcup_{\substack{j = 1 \\ j \neq i}}^{n_p} \left({R}^p_i \cap {R}^p_j \right).
\end{align*}
 ${R}^p_{i,\te{col}}$ is the set of all points in ${R}^p_i$ that can result in a collision between $p_i$ and the members of its team. The shaded region in Fig. \ref{fig:Arena} depicts ${R}^p_{i,\te{col}}$ for $p_i$. To make the game more challenging and to add obstacle avoidance, we assume that some point obstacles are placed in the arena. Let 
\[
\mc{O} = \{z^{\te{obs}}_1,\ldots,z^{\te{obs}}_{n_{\te{obs}}-1}\}
\]
be the set of obstacle locations. 
In Fig. \ref{fig:Arena}, the obstacles are represented by asterisks. 

For player $p_i$, we combine the sets of points of possible collisions with other players and with obstacles as follows. 
\begin{align}\label{set:Avoid}
{R}^p_{i,\te{avoid}} ={R}^p_{i,\te{col}} \cup (R^p_i(k)\cap \mc{O})
\end{align}
One approach to guarantee collision avoidance is to limit the effective reachable set of $p_i$ to ${R}^p_i -{R}^p_{i,\te{avoid}}$, which is the set of points of ${R}^p_i$ that are not included in ${R}^p_{i,\te{avoid}}$. To avoid the points in ${R}^p_{i,\te{avoid}}$, we compute $X^p_{i,\te{avoid}}$ and $Y^p_{i,\te{avoid}}$. These are sets of collision avoidance planes  along $x$ and $y$ direction such that avoiding these entire planes guarantee that $z^p_i \in {R}^p_i -{R}^p_{i,\te{avoid}}$ in the next time step. We explain collision avoidance planes through examples in Fig. \ref{fig:collision}. In both the cases, the shaded regions are the sets where collisions can occur. In Fig. \ref{fig:collision}(a), if $p_i$ avoids the entire plane $x = x^p_i + 1$ and $p_j$ avoids the entire plane $x=x^p_j-1$, then $p_i$ and $p_j$ cannot collide at time $k+1$. In this case
\begin{align*}
X^p_{i,\te{avoid}} &= \{x^p_i + 1\},~Y^p_{i,\te{avoid}} = \{\}\\
 X^p_{j,\te{avoid}} &= \{x^p_j - 1\},~Y^p_{j,\te{avoid}} = \{\}
\end{align*}
Similarly, the avoidance planes in Fig. \ref{fig:collision}(b) are
\begin{align*}
X^p_{i,\te{avoid}} &= \{x^p_i + 1\},~Y^p_{i,\te{avoid}} = \{y^p_i + 1\}\\
X^p_{j,\te{avoid}} &= \{x^p_j - 1\},~Y^p_{j,\te{avoid}} = \{\} \\
X^p_{l,\te{avoid}} &= \{\}, Y^p_{l,\te{avoid}} = \{y^p_l - 1\}
\end{align*}
For $u_x^{\max} = u_y^{\max} =1$, the sets $X^p_{i,\te{avoid}}$ and $Y^p_{i,\te{avoid}}$ can be computed easily.

Next, we formulate the problem form the perspective of defense team . In the game setup, we assume that at time $k$ each defender knows the current locations of all the attackers. However, any mobility strategy for the defense team inherently depends on the strategy of the attack team, which is unknown to the defenders. Therefore, we implement an MPC based online optimization strategy, in which the defenders assume a mobility model for attackers over a prediction horizon. At each time $k$, the online optimization problem has the following structure. 
\begin{align} \label{prob:Problem_Online}
\min_{\te{u}^d \in \mc{U}^d} ~~~~&J(\te{z}^{d},\te{u}^d,(\te{z}^a)^+,\te{z}^f,\te{z}^{\te{obs}})\nonumber  \\
\text{s.t.} ~~~~~ &(\te{z}^{d})^+  = \te{z}^d + \te{u}^d . \tag{$\mc{P}2$} 
\end{align}
where $\mc{U}^d = \prod\limits_{i = 0}^{n_d-1} U^d_i$.  

In this problem formulation, $\te{z}^a$ and $\te{z}^d$ are the location vectors of the offense and defense teams at time $k$, and  $(\te{z}^a)^+$ is the location vector of offense at time $k+1$. To solve \ref{prob:Problem_Online}, defenders still need to know  $(\te{z}^a)^+$, which cannot be known at current time. Therefore, defenders assume a mobility model for attackers. The assumed model can be as simple as a straight line path form the current location of an attacker to the defense zone. The model can also be a more sophisticated like a feedback strategy as presented in \cite{chasparis2008}. 

The cost function $J$ in \ref{prob:Problem_Online} is 
\begin{align*}
J(\te{z}^{d},\te{u}^d,(\te{z}^a)^+,\te{z}^f, \te{z}^{\te{obs}})  = \sum_{i=0}^{n_d-1}J_i(\te{z}^{d},\te{u}^d,(\te{z}^a)^+,\te{z}^f, \te{z}^{\te{obs}}) 
\end{align*}
The total cost is the sum of the costs of the individual agents. The local cost of an agent is
\begin{multline}\label{eq:J_total}
J_i =  \alpha_i^f  J_i^f(z_i^d,u^d_{i},\te{z}^f) +  \alpha_i^a J_i^a(z_i^d,u^d_{i},(\te{z}^a)^+)+ J_i^d(\te{z}^{d}, \te{u}^d)+ \\J_i^{\te{avoid}}(\te{z}^{d},u^d_i,\te{z}^{\te{obs}}) + J_i^{\te{mob}}(u^d_i) .
\end{multline}
\begin{figure}[t] 
	\centering
	\includegraphics[trim= 0cm 0cm 0cm 0cm,clip, scale=0.5]{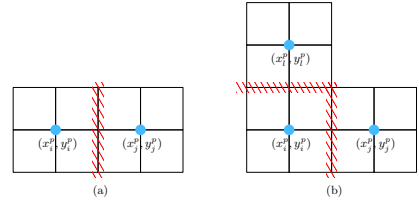}
	\caption{Collision avoidance planes.} 
	\label{fig:collision}
\end{figure}
To avoid notational clutter, we will ignore function arguments unless necessary. The terms comprising the cost function are defined as follows. 
\begin{align}
J_i^f&=  \sum_{h = 0}^{n_f-1}w^f_{ih} d((z_i^d)^+,z_h^{{f}}) \label{eq:Jf}\\
J_i^a&=\sum_{g = 0}^{n_a-1}w^a_{ig} d((z_i^d)^+,(z_g^a)^+), \label{eq:Ja}\\
J_i^d &=  \sum_{j = 0}^{n_d-1}w^d_{ij} d((z_i^d)^+,(z_j^d)^+), \label{eq:Jd}\\
J_i^{\te{avoid}} &= \sum_{c_x \in {X}^d_{i,\te{avoid}}}  d_{x,\te{avoid}}((z_i^d)^+,c_x)\nonumber \\
&~~~~~~~ +\sum_{c_y \in {Y}^d_{i,\te{avoid}}}  d_{y,\te{avoid}}((z_i^d)^+,c_y)\label{eq:Javoid}\\
J_i^{\te{mob}}(u^d_i) &= w^u_i(|u^d_{x,i}|^2+|u^d_{y,i}|^2) \label{eq:Jmobility}
\end{align}
In the cost functions, $w^f_{ih}$, $w^a_{ig}$, $w^d_{ij}$, and $w^u_i$ are non-negative weights. The function $ d(z_i,z_j)$ is a distance measure between points $z_i$ and $z_j$. It can be either of the following two functions.
\begin{align}\label{eq:distance}
d(z_i,z_j) &= (x_i-x_j)^2+(y_i-y_j)^2\nonumber\\
d(z_i,z_j) &= |x_i-x_j|+|y_i-y_j|.
\end{align}
The functions $d_{x,\te{avoid}}$ and $d_{y,\te{avoid}}$ are
\begin{align*}
d_{x,\te{avoid}}(z_i^d,c_x) &= \zeta_1e^{-\zeta_2 (x_i^d - c_x)^2}\\
d_{y,\te{avoid}}(z_i^d,c_y) &= \zeta_1e^{-\zeta_2 (y_i^d - c_y)^2}
\end{align*}

with $\zeta_1 \geq 1$ and $\zeta_2\ge1$. 

The local cost of each defender has five components, each of which is a potential function with the minimum value at the desired location. The first and the second terms jointly model the behavior of a defender. The term $J_i^f$ in  (\ref{eq:Jf}) models defensive behavior in which $d_i$ stays close to the defense zone to protect it. The constant weight $w^f_{ih}\geq 0$ is the strength of attractive force between $d_i$ and the point $z^{f}_h$ in the defense zone. The cost term $J_i^f$  is minimized when $z^d_i$ is equal to a weighted average of the points in the defense zone. We assume that each defender $d_i$ is assigned the responsibility of a subset $D_i$ of the defense zone $D$, where 
\[
D_i \subseteq D~~~\te{ and } ~~~\bigcup_{i=1}^{n_d} D_i =D.
\]
The weights are assigned as follows. 
\begin{equation}
w^f_{ih} =
\begin{cases}
\frac{1}{|D_i|} &\quad z^f_h \in D_i \\
0 &\quad \te{otherwise}
\end{cases}
\end{equation}

The term $J^a_i$ defined in  (\ref{eq:Ja}) models attacking behavior of the defenders. In this mode, the defenders actively pursue the attackers and try to capture them before they reach the defense zone. The function $J_i^a$ is a weighted sum of square of the distances between $d_i$ and the locations of the attackers at the next time step. For the simulations in the next section, we assume that defender $d_i$ only pursues the attacker that is closest to ${D}_i$. Let
\begin{align*}
\delta(d_i,a_g) &= \min \{{d(z^{f}_h,z^a_g}): z^f_h\in D_i\}\\
\delta(d_i) &= \min\{\delta(d_i,a_g) : g \in \{1,\ldots,n_a\}\}.
\end{align*}
Here $\delta(d_i,a_g)$ is the minimum Manhattan distance of attacker $a_g$ from $D_i$ and $\delta(d_i) $ is the minimum of the distances of all the attackers from $D_i$. Then 
\begin{equation}\label{eq:weightAttacker}
w^a_{ig} =
\begin{cases}
c &\quad \delta(d_i,a_g) =  \delta(d_i) \\
0 &\quad \te{otherwise}
\end{cases}
\end{equation}
where $c$ is a scalar. In case of a tie, $d_i$ selects an attacker randomly and starts pursuing it. 

The behavior of each defender can be selected to be a  combination of these two terms by tuning the parameters $\alpha_i^a$ and $\alpha^f_i$ such that 
\[
\alpha_i^a + \alpha_i^f = 1.
\]
The values $\alpha^a_i= 1$ or $\alpha^f_i = 1$ corresponds to purely attacking or defensive behaviors for $d_i$. If these parameters are constant, the behavior of the defenders remain the same through out the game. We can also have an adaptive strategy based on feedback for adjusting the behavior of each defender. 
Let $\delta_{\text{th}}$ be a threshold at which the  behavior of a defender switches between attack and defense modes. Let $\alpha^a_{\text{nom}}$ and $\alpha^f_{\text{nom}}$ be the nominal weights assigned to $J_i^a$ and $J_i^f$ respectively at $\delta(d_i,k) = \delta_{\text{th}} $ such that 
\[
\alpha^a_{\text{nom}} + \alpha^f_{\text{nom}} = 1.
\]
Then 
\begin{align}\label{eq:alpha}
\alpha_i^a &=\frac{ \alpha^a_{\text{nom}} e^{\beta(\delta_{\te{th}} -\delta(d_i))}}{\alpha^a_{\text{nom}} e^{\beta(\delta_{\te{th}} -\delta(d_i))} + \alpha^f_{\text{nom}}}, \\
\alpha_i^f &=\frac{ \alpha^f_{\text{nom}}}{\alpha^a_{\text{nom}} e^{\beta(\delta_{\te{th}} - \delta(d_i))} + \alpha^f_{\text{nom}}}.\nonumber
\end{align}
where $\beta \in [0,1]$ is a constant value. 

For defender $d_i$, if $\delta(d_i)= \delta_{\te{th}}$, the parameters $\alpha_i^a$ and $\alpha_i^f$ are equal to their nominal values. If an attacker gets closer to $D_i$ than $\delta_{\te{th}}$, i.e.,  $\delta(d_i) < \delta_{\te{th}}$, the value of $\alpha_i^a$ increases exponentially and the value of $\alpha^f$ decreases. Thus, as the attackers move towards $D_i$, the weight assigned to $J_i^a$ increases, and the behavior of $d_i$ shifts towards attacking mode . However, if the attackers are not close to the $D_i$, i.e., $\delta(d_i) >\delta_{\te{th}}$, then the value of $\alpha^a$ reduces exponentially and the behavior of $d_i$ becomes more defensive. 

The third term  $J_i^d$ defined in  (\ref{eq:Jd}) generates cohesion among the defenders. Minimizing $J_i^d$ drives $d_i$ towards the weighted average of the locations of all the other defenders at time $k+1$. We assume that the weights $w^d_{ih}$ are positive, i.e., 
\[
w^d_{ih} > 0 \text{ for all } i , h \text{ in } \{1,2,\ldots,n_d\}.
\]
$J^d_i$ depends on the next locations of all the defenders. We assume that the each defender knows that current locations of all of its teammates. However, it does not know the behavior parameters of other defenders, i.e., $\alpha_i^a$ and $\alpha_i^f$ are private parameters of each player. Therefore, we need to implement a distributed optimization algorithm to minimize $J^d_i$. We will show through simulations that the proposed algorithm Alg. \ref{alg:DistOpt} can be used effectively to minimize $J^d_i$. 

The fourth term $J_i^{\te{avoid}}$ defined in  (\ref{eq:Javoid}) guarantees obstacle and collision avoidance. The function $d_{x,\te{avoid}}((x^d_i)^+,c_x)$ is maximum when $(x^d_i)^+ = c_x$, where $c_x \in X^d_{i,\te{avoid}}$. Similarly, $d_{y,\te{avoid}}((y^d_i)^+,c_y)$ is maximum when $(y^d_i)^+ = c_y$, where $c_y\in Y^d_{i,\te{avoid}}$. By selecting $\zeta_1$ large enough, we can guarantee that $d_i$ avoids the planes in $X^d_{i,\te{avoid}}$ and $Y^d_{i,\te{avoid}}$, which ensures that it avoids $R^d_{i,\te{avoid}}$. Thus, minimizing (\ref{eq:Javoid}) guarantees collision and obstacle avoidance. The purpose of $\zeta_2$ is to control the region of influence of this barrier potential. Finally, the fifth term $J_i^{\te{mob}}$ in (\ref{eq:Jmobility}) is the mobility cost of $d_i$. 

Problem \ref{prob:Problem_Online} is a combinatorial optimization problem because the set of inputs is discrete. We will now prove that the cost $J$ in  (\ref{eq:J_total}) is submodular and \ref{prob:Problem_Online} is a submodular minimization problem. 

\begin{thm}\label{thm:Cost_Submodular}
{\it Problem \ref{prob:Problem_Online} with the cost function defined in (\ref{eq:J_total})-(\ref{eq:Jmobility}) is a submodular minimization problem over 
\[
\mc{U} = \prod_{i=0}^{n_d-1} U_{x,i}\times U_{y,i} 
\]
where $U_{x,i}$ and $U_{y,i}$ are equal to 
\[
U =\{-u_{\max},\ldots,0,\ldots,u_{\max}\} \text{ for all } i.
\]
}
\end{thm} 
\begin{proof}
Since $U$ is a subset of $\field{Z}$, we can use the criterion in  (\ref{eq:antitoneInt}) to verify submodularity of the the cost function. From (\ref{eq:J_total}), the cost of each agent $J_i$ is a summation of five terms. We will show that each of these terms is submodular. Then, using the property that sum of submodular functions is also submodular, we prove the theorem.

The terms $J^f_i$, $J^a_i$, and $J^d_i$ are weighted sums of the distance functions in (\ref{eq:distance}). We verify that 
\[
d(\te{z}) = (x_i - x_j )^2 + (y_i - y_j)^2, 
\]
is submodular for any $\te{z} = (x_i,y_i,x_j,y_j)$. To prove that $d(\te{z})$ is submodular, we need to show that
\[
d(\te{z}+\te{e}_p + \te{e}_q) - d(\te{z} + \te{e}_q) \leq d(\te{z} + \te{e}_p) - d(\te{z} ).
\]
where $\te{e}_p$ and $\te{e}_q$ are unit vectors of dimension four. 

The first scenario is that both $\te{e}_p$ and $\te{e}_q$ increment either the $x$ or the $y$ components in $\te{z}$. Without loss of generality, we assume that the $x$ components are incremented, i.e., $p = 1$ and $q = 3$. Then, 
\begin{multline*}
[d(\te{z}+\te{e}_p + \te{e}_q) - d(\te{z} + \te{e}_q)] - [d(\te{z} + \te{e}_p) - d(\te{z} )] = \\
  -(1-2(x_i - x_j))  - (1 + 2(x_i - x_j)) = -2
\end{multline*}
The second scenario is that out of $p$ and $q$, one corresponds to an $x$ component and the other corresponds to a $y$ component. Let $p = 2$ and $q = 3$. Then, 

\begin{multline*}
[d(\te{z}+\te{e}_p + \te{e}_q) - d(\te{z} + \te{e}_q)] - [d(\te{z} + \te{e}_p) - d(\te{z} )] = \\
(1+2(y_i - y_j))  - (1 + 2(y_i - y_j)) = 0
\end{multline*}

Thus, the condition in (\ref{eq:antitoneInt}) is satisfied for all possible scenarios, which proves that the function $d(\te{z})$ is submodular. The same sequence of steps can be followed to verify the submodularity of Manhattan distance. 

The function $J^i_{\te{avoid}}$ is the summation of the terms $d_{x,\te{avoid}}$ and $d_{y,\te{avoid}}$. Since each of these terms is only a function of single decision variable, the second order comparison condition in (\ref{eq:antitoneInt}) will be satisfied with equality. The same argument is valid for $J^{\te{mob}}$.
Since all the functions in $J_i$ are submodular, $J_i$ is submodular for all $i$, which concludes the proof.

\end{proof}
	
\subsection{Discussion}
\begin{enumerate}
\item To solve \ref{prob:Problem_Online}, defender $d_i$ can minimize the terms $J_i^{\te{avoid}}$, $J^a_i$, and $J^f_i$ locally without coordinating with its team members. All of these terms only depend on the quantities that are known to $d_i$ at time $k$, i.e., locations of obstacles, current locations of other defenders, locations of attackers, and the locations of the defense zone. Although $d_i$ does not know the exact locations of the attackers, it uses an estimate of these locations based on the current locations of the attackers and the assumed mobility model.

\item The term $J_i^d$ depends on the future locations of all the defenders and cannot be minimized locally. Thus, we need a distributed optimization algorithm to solve \ref{prob:Problem_Online}. We assume that each defender can only communicate with a subset of the members of the defense team and the adjacency matrix representing the communication network topology satisfies the assumptions presented in Section \ref{sec:Distributed Submodular Minimization}. Since \ref{prob:Problem_Online} is an online optimization problem, the defenders need to execute Alg. \ref{alg:DistOpt} at each decision time to compute an update action. 

\item By avoiding the planes in $X^d_{i,\te{avoid}}$ and $Y^d_{i,\te{avoid}}$, collision and obstacle avoidance is guaranteed. Although, this approach for avoiding obstacles and collisions can be  conservative, the objective here is to demonstrate that we can generate repulsive forces and implement avoidance behaviors using submodular functions. For practical scenarios, we can always design sophisticated protocols by assigning priorities to agents in the case of deadlocks.   

\item Submodularity yields significant savings in terms of the number of computations as $n_d$ increases. For example, if $n_d = 4 $ and $u_{\max} = 1$, the dimension of the problem for each agent is $9^4=6561 $, since $| U\times U| = 9$. However, the subgradient can be computed in $r\log(r)+r$ computations. For $n_d = 4$ and $|U| = 3$, we have $r = 16$ from (\ref{eq:r}). Since we are solving an online optimization problem under uncertainties, executing 15 to 20 iterations of the distributed optimization algorithm should be sufficient to compute an approximate solution for a problem of this size.  
\end{enumerate}


\begin{algorithm}[t!]
\caption{\bf Online Distributed Submodular Minimization }\label{alg:DistOpt_Online}
\begin{algorithmic}[1]
\Statex At each decision time $k$, $d_i$ needs to perform the following steps:
\For {$k=0$ to $T-1$}
\State Apply Alg. \ref{alg:DistOpt} with $\hat{\te{u}}^i$ as output. 
\State Update the state 
\begin{equation*}
z^d_i(k+1) = z^d_i(k) + \hat{u}^i_{i}. 
\end{equation*}

\EndFor 
\end{algorithmic}
\end{algorithm}

\section{Simulation}\label{sec:Simulation}
We simulated the capture the flag game with the following setup.  The size of the grid was $20\times 20$ and the game was played over a time interval of length $T = 40$. The defense zone was located at the top of the field.
The number of players in both the teams was four, i.e., $n_d = n_a = 4$. 
There were six point obstacles placed in the field. 
%
\begin{figure*}[t] 
	\centering
	\subfigure[Playing arena ]
	{
		\includegraphics[trim= 3cm 0cm 0cm 0cm,clip, scale=0.17]{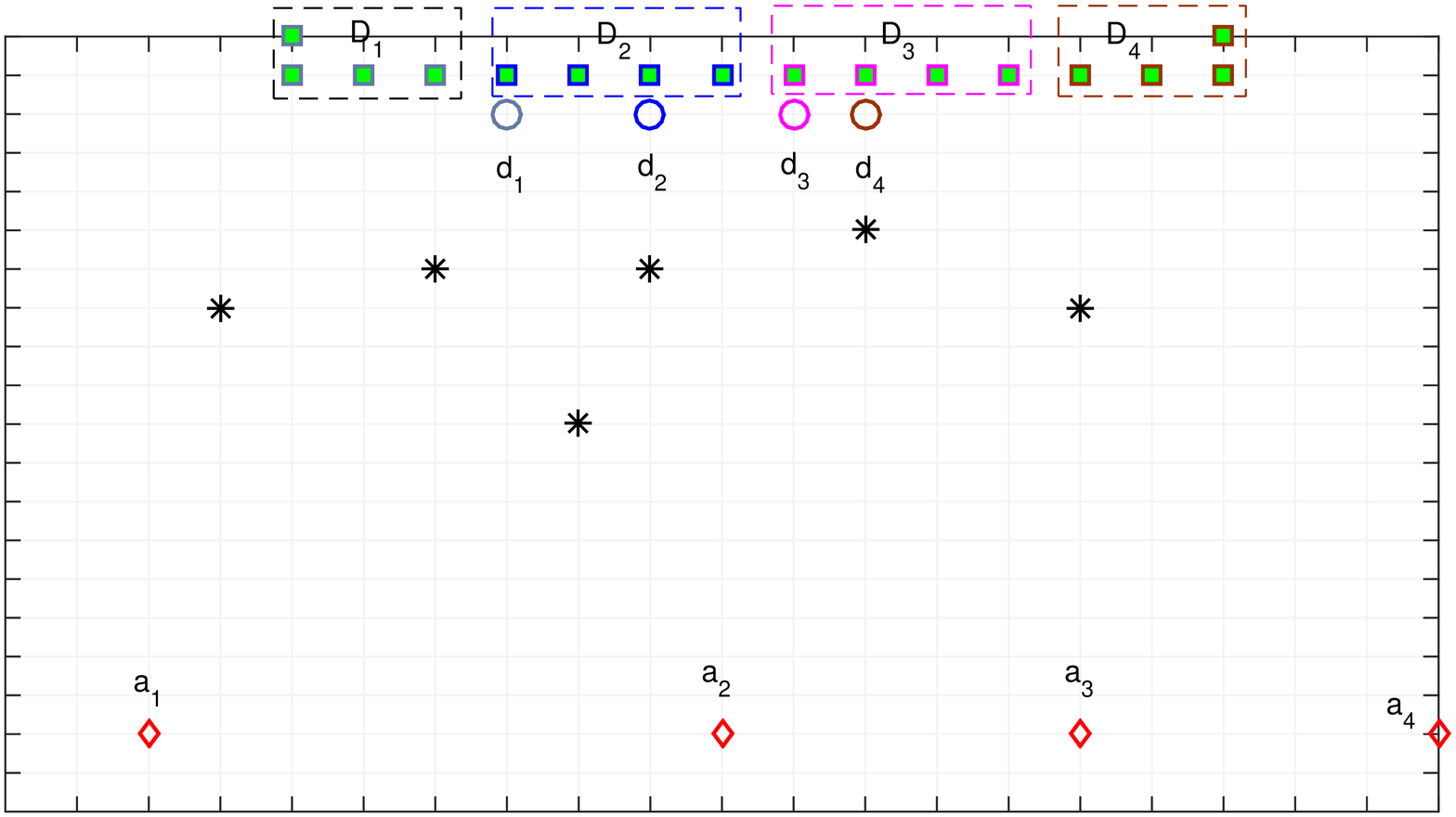}\label{fig:Layout}
	}
	\subfigure[$\delta_{\te{th}} = 20$ ]
	{
		\includegraphics[trim= 3cm 0cm 0cm 0cm,clip, scale=0.17]{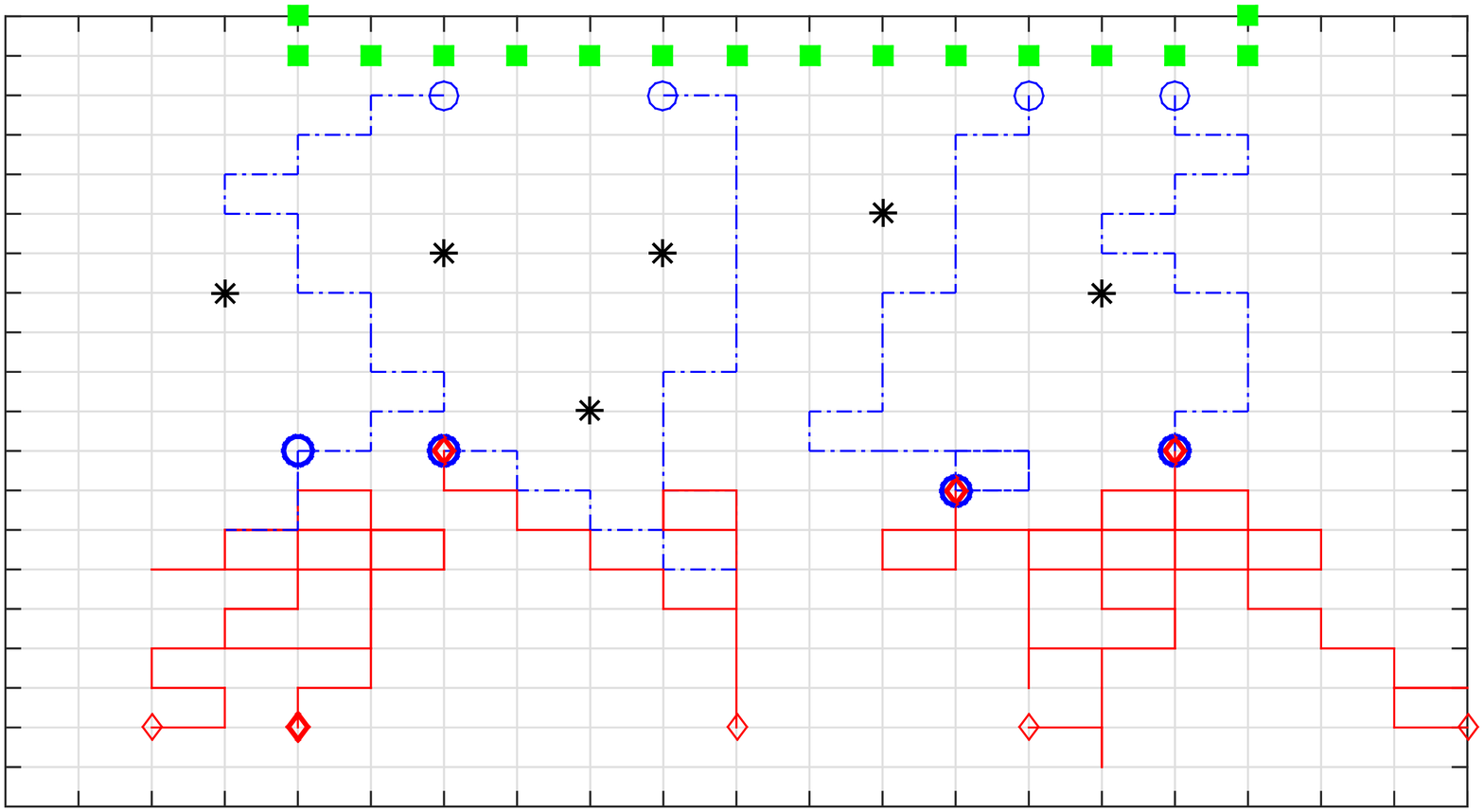}\label{fig:sim}
	}
	\subfigure[$\delta_{\te{th}} = 15$.  ]
	{
		\includegraphics[trim= 3cm 0cm 0cm 0cm,clip, scale=0.17]{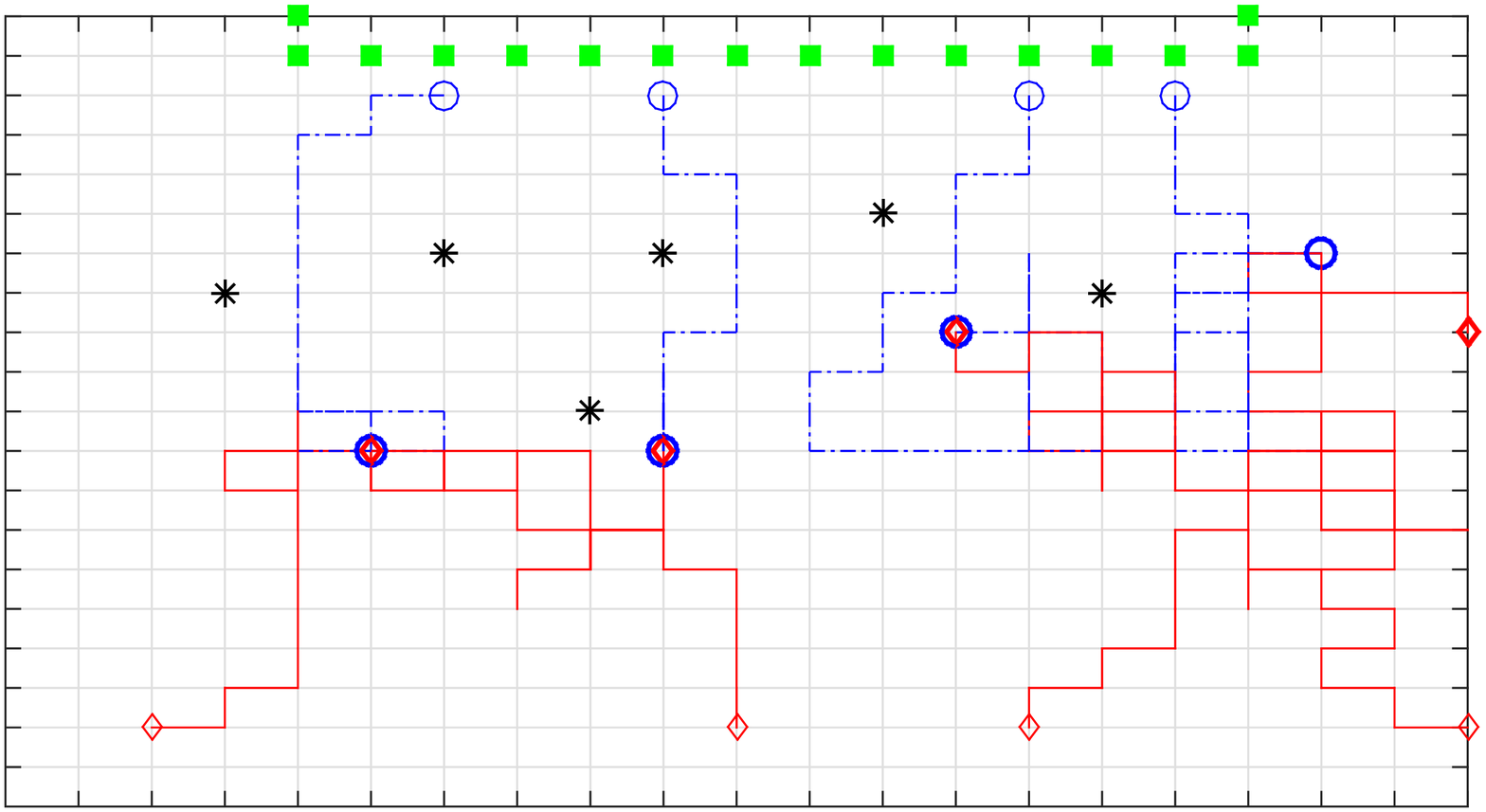}\label{fig:sim1}
	}\\
	\subfigure[$\delta_{\te{th}} = 10$.  ]
	{
		\includegraphics[trim= 3cm 0cm 0cm 0cm,clip, scale=0.17]{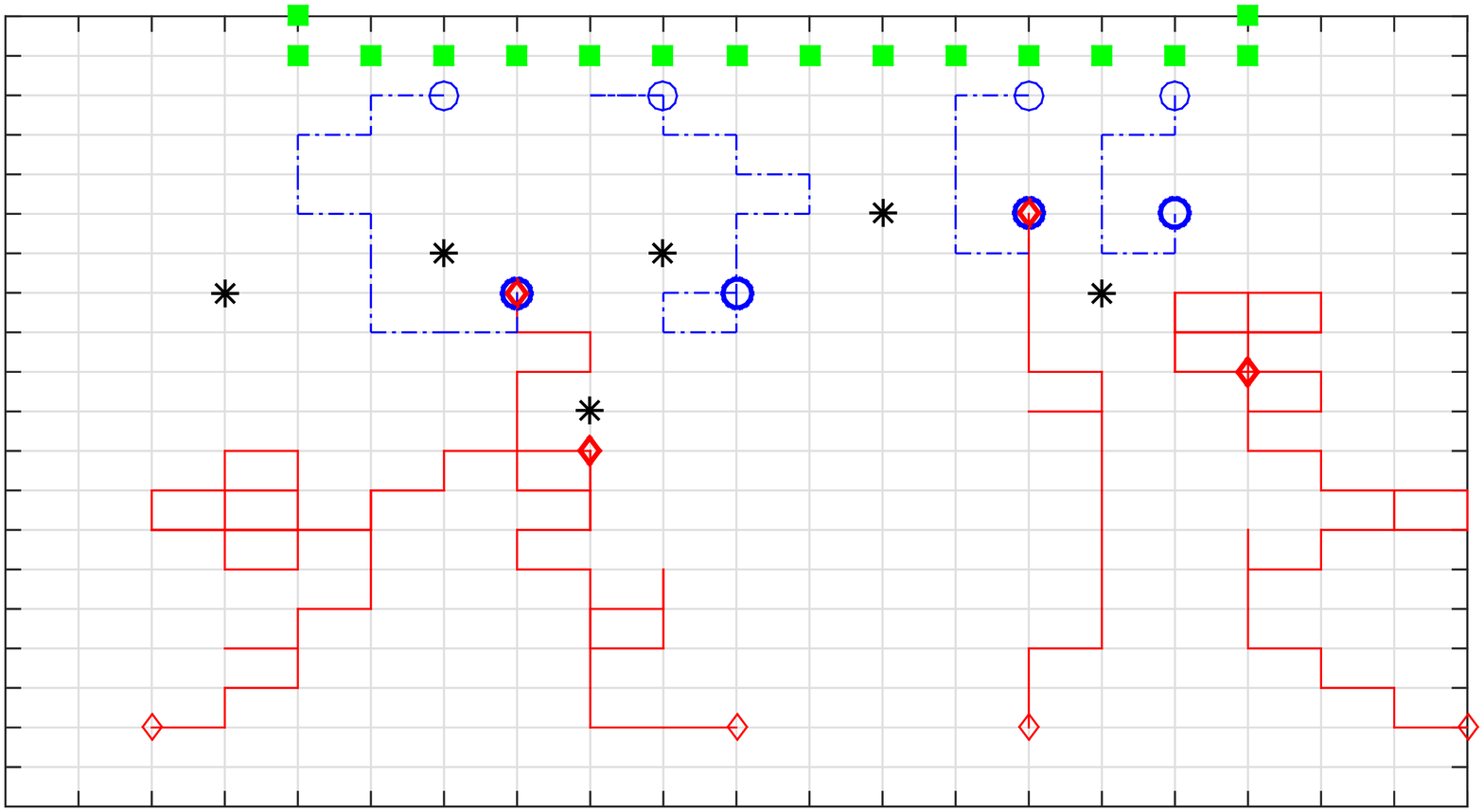}\label{fig:sim2}
	}
	\subfigure[$\delta_{\te{th}} = 5$.  ]
	{
		\includegraphics[trim= 3cm 0cm 0cm 0cm,clip, scale=0.17]{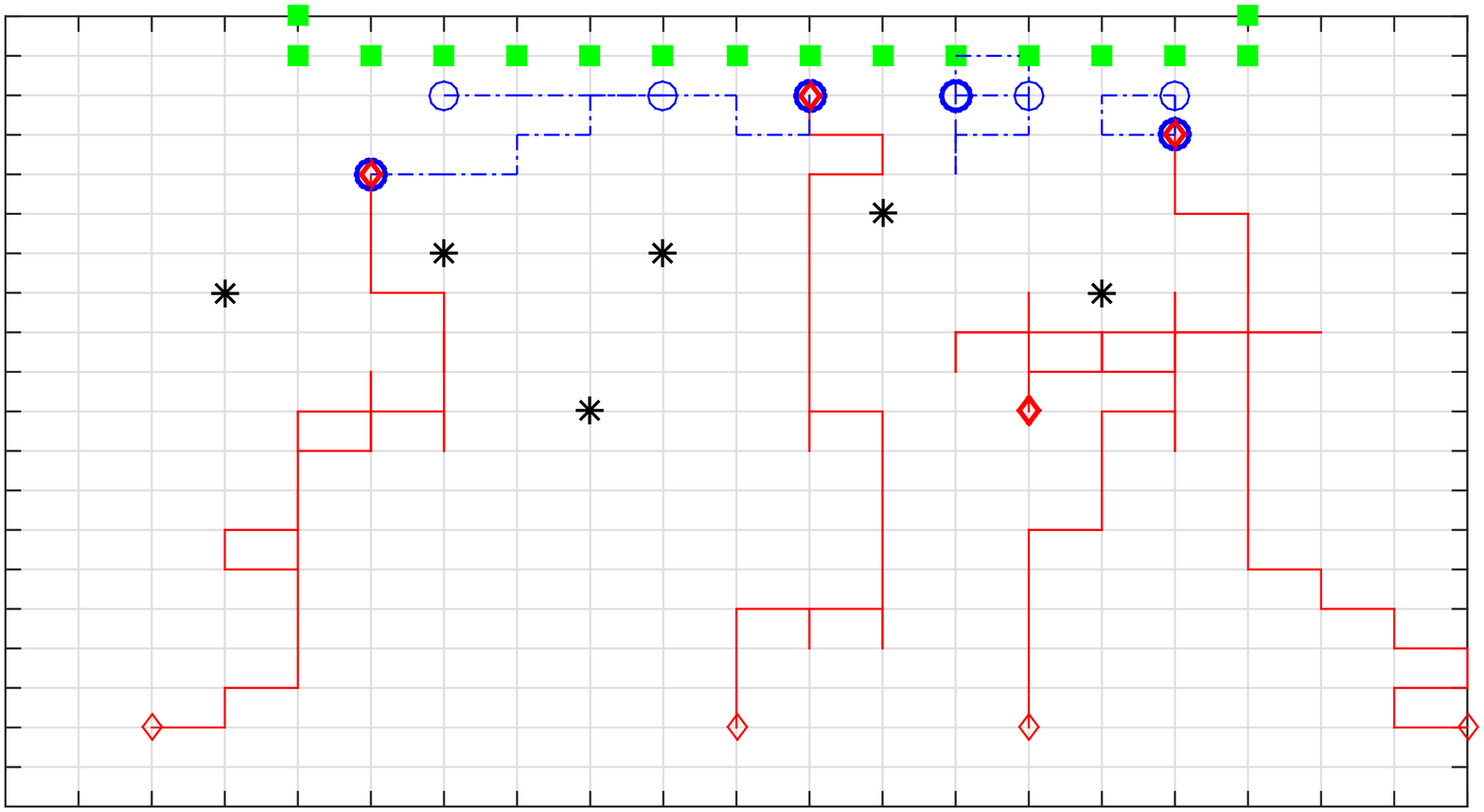}\label{fig:sim3}
	}
	\subfigure[$\delta_{\te{th}} = (20,8,8,20)$ ]
	{
		\includegraphics[trim= 3cm 0cm 0cm 0cm,clip, scale=0.17]{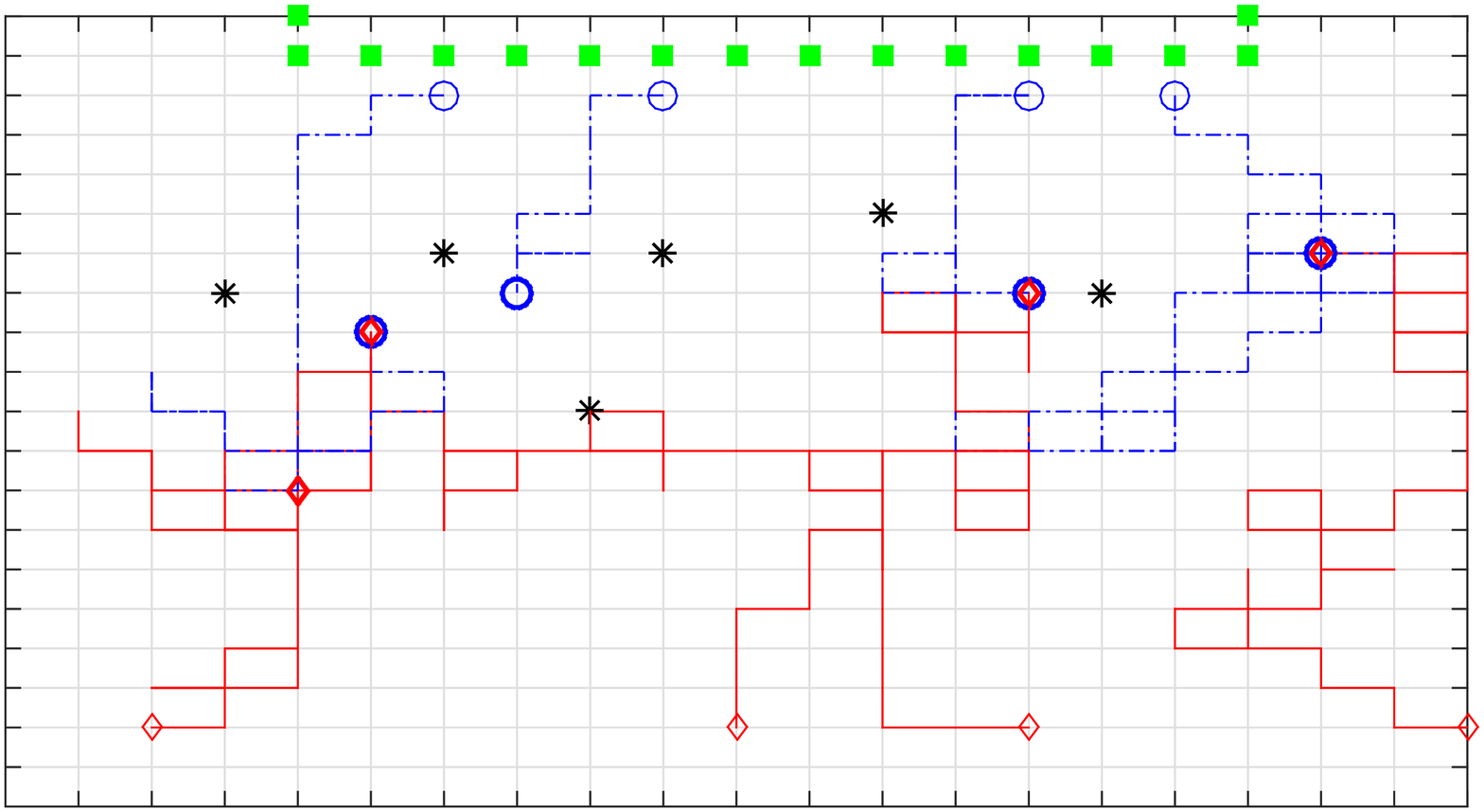}\label{fig:sim4}
	}
	\caption{Layout of the playing arena and trajectories of the offense and defense teams for different values of $\delta_{\te{th}}$ in  (\ref{eq:alpha}).} 
	\label{fig:Simulations}
\end{figure*}
The detailed layout of the field with the locations of the defense zone, attacker, defenders, and the obstacles is presented in Fig. \ref{fig:Layout}. 
The defense zone is the set of squares at the top of the field. The obstacles are represented by asterisks, attackers by diamonds, and defenders by circles. The responsibility set of each defender $d_i$ was $D_i$ and is shown in the figure. 

The parameters in $J^{\te{avoid}}_i$ were set as $\zeta_1 = 200$ and $\zeta_2 = 5$. The weights in (\ref{eq:weightAttacker}) for $J^a_i$ was $c =20$ for each defender. For cohesion among the defenders in $J^d_i$, the following weights were used
\[ \small
 W^d   =
\begin{bmatrix}
         0.0 & 0.5 & 0.1 & 0.01\\ 
          0.5 & 0.0 & 0.1 & 0.01 \\
          0.01 & 0.1& 0.0 & 0.5 \\ 
          0.01&  0.1 & 0.5  &0.0 
\end{bmatrix}
\]
where $W^d_{ij} = w^d_{ij}$. 
To switch between attacking and defense modes, we selected $\alpha_{\te{nom}}^f = 0.9$ and  $\alpha_{\te{nom}}^f = 0.1$ and $\beta = 0.7$. The simulations were performed with Manhattan distance for the function $d(z_i,z_j)$ as defined in (\ref{eq:distance}) and for four different values of threshold distance, $\delta_{\te{th}} \in  \{5,10,15,20\}$. We also simulate a scenario with different value of $\delta_{\te{th}}$ for each defender.

The defenders assumed that the attackers always try to minimize their distance form the defense zone. However, the actual strategy of the attackers was based on feedback that depended on their distance from the defenders. Each attacker had two basic modes: attack base and avoid defender. It adjusted the weights assigned to each of these modes depending on its minimum distance from the defenders. 

At each decision time, attacker $i$ computed two positions. To enter the defense zone, it computed the location in its neighborhood that minimized its distance from the defense zone. To avoid defenders, it also computed the location that maximized its distance from the nearest defender. Let $\eta_{i,\te{base}}$ be the weight assigned to attack base mode and $\eta_{i,\te{avoid}}$ be the weight assigned to avoid defender mode. Let $\eta_{\te{base}}^{\te{nom}}$ and $\eta_{\te{avoid}}^{\te{nom}}$ be the nominal values if the minimum distance between an attacker and the defenders was equal to some threshold value $\Delta_{\te{th}}$. Let $\Delta^a_i(k)$ be the minimum distance between attacker $i$ and the defenders at time $k$. Then 

\begin{align*}
\eta_{i,\te{avoid}}(k) &= \frac{\eta_{\te{avoid}}^{\te{nom}}e^{\kappa(\Delta_{\te{th}} - \Delta^a_i(k)}}{\eta_{\te{base}}^{\te{nom}} + \eta_{\te{avoid}}^{\te{nom}}e^{\kappa(\Delta_{\te{th}} - \Delta^a_i(k)}},\\
\eta_{i,\te{base}}(k) &= \frac{\eta_{\te{base}}^{\te{nom}}}{\eta_{\te{base}}^{\te{nom}} + \eta_{\te{avoid}}^{\te{nom}}e^{\kappa(\Delta_{\te{th}} - \Delta^a_i(k)}}.
\end{align*}
where $\kappa \in [0,1]$ is a constant value. If $\Delta^a_i(k) < \Delta_{\te{th}}$, the value of $\eta_{i,\te{avoid}}(k)$ increases because a defender is closer than the threshold value. However, as $\Delta^a_i(k)$ increases, $\eta_{i,\te{avoid}}(k)$ keeps on decreasing. Thus, at each decision time $k$, the attacker decides to attack the base with probability $\eta_{i,\te{base}}(k)$ and avoid defenders with probability $\eta_{i,\te{avoid}}(k)$. In the simulation setup, we selected $\eta_{\te{avoid}}^{\te{nom}} = 0.7$, $\eta_{\te{base}}^{\te{nom}} = 0.3$, $\Delta_{\te{th}} = 4$ and $\kappa = 0.9$. 


Finally, for the distributed optimization algorithm, we assume that the communication network topology is a line graph and the adjacency matrix has the following structure. 
\[ \small
 A  = 
\begin{bmatrix}
         0.7 & 0.3  & 0 & 0\\ 
          0.3 & 0.6 & 0.1 & 0 \\
          0 &0.1 & 	0.6 & 0.3  \\ 
          0 & 0 &  0.3 & 0.7 \\
\end{bmatrix}
\]
The distributed subgradient algorithm was executed for $\te{iter} = 20$ iterations with $\gamma =0.1 $ and $\hat{t} = 0.7$. The simulation results are presented for four values of $\delta_{\te{th}}$, which controlled the transition of defenders behavior from defense to attack in  (\ref{eq:alpha}). In all the simulations, the behavior of the attackers was aligned with the values set for $\eta_{\te{avoid}}^{\te{nom}}$ and $\eta_{\te{base}}^{\te{nom}}$. The attackers had more emphasis on avoiding the defenders than capturing the base. Consequently, none of the attackers could enter the defense zone. However, the defenders were unable to capture all the attackers as well.

The effect of decreasing the value of $\delta_{\te{th}}$ can be observed by comparing the trajectories in Figs. \ref{fig:sim}-\ref{fig:sim3}. With $\delta_{\te{th}} = 20$, the behavior of the defense team was set to be attacking, which is evident form Fig. \ref{fig:sim}. The defenders left the base in pursuit of the attackers and were able to capture three of them. As $\delta_{\te{th}}$ is reduced, the defensive behavior becomes more and more prominent. In Fig. \ref{fig:sim1} for $\delta_{\te{th}}=15$, the defenders left the base area in pursuit of the attackers but were a little restrictive then the case with $\delta_{\te{th}} = 20$ in Fig. \ref{fig:sim}. The behavior of the defenders became more restrictive in Fig. \ref{fig:sim2} when $\delta_{\te{th}} = 10$. 
With $\delta_{\te{th}} = 5$, the behavior of the defenders was set to be defensive. Therefore, we can observe from Fig. \ref{fig:sim3} that all the defenders remained close to their assigned base area. Finally, we simulated the game with different $\delta_{\te{th}}$ for each defenders. From Fig. \ref{fig:sim4}, we can observe that the defenders at the flanks were attacking and the center players were more defensive and guarded the defense zone. 

In all the simulations, we can observe that the defenders managed to avoid collisions among themselves and with the obstacles. Thus, the proposed framework for the distributed minimization of submodular functions generated effective trajectories for our problem even though our problem was defined over partially ordered sets. 

\section{conclusion}
We presented a framework for the distributed minimization of submodular functions over lattices of chain products. For this framework, we established a novel connection between a particular convex extension of submodular functions and distributed optimization of convex functions. This connection proved to be effective because that convex extension of submodular functions could be computed efficiently in polynomial time. Furthermore, the solution to the original submodular problem was directly related to the solution of the equivalent convex problem. 

We also proposed a novel application domain for submodular function minimization, which is distributed motion coordination over discrete domains. We demonstrated through an example setup that we can design potential fields over state space based on submodular functions. We showed that we can achieve certain desired behaviors like cohesion, go to goal, collision avoidance, and obstacle avoidance by driving the agents towards the minima of these potential fields. Finally, we verified through simulations that the proposed framework for distributed submodular minimization can efficiently minimize these submodular potential fields online in a distributed manner.

\bibliography{IEEEabrv,CDC2017}

\end{document}